\documentclass{article}
\usepackage[utf8]{inputenc}
\usepackage{amsmath,amssymb,amsthm}
\usepackage{mathtools}
\usepackage{enumerate}
\usepackage{tikz}
\usetikzlibrary{arrows.meta,calc}
\usepackage{lipsum}

\usepackage{graphicx} 
\usepackage{tikz-3dplot}
\usepackage{comment}

\usepackage{xcolor}
\definecolor{darkgreen}{rgb}{0,0.7,0}

\usepackage{xcolor}
\usepackage{etoolbox}
\usepackage[
  colorlinks=true,
  linkcolor=darkgreen,
  citecolor=darkgreen,
  urlcolor=black
]{hyperref}

\makeatletter
\patchcmd{\tableofcontents}
  {\@starttoc{toc}}
  {\begingroup\hypersetup{linkcolor=black}\@starttoc{toc}\endgroup}
  {}{}
\makeatother

\newtheorem{theorem}{Theorem}[section]
\newtheorem{lemma}[theorem]{Lemma}
\newtheorem{definition}[theorem]{Definition}

\newtheorem{notation}[theorem]{Notation}
\newtheorem{remark}[theorem]{Remark}
\newtheorem{corollary}[theorem]{Corollary}
\newtheorem{theorem/definition}[theorem]{Satz/Definition}
\newtheorem{proposition}[theorem]{Proposition}

\newcommand{\R}{\mathbb R}

\newcommand{\N}{\mathbb{N}}

\newcommand{\C}{{\cal C}}

\newcommand{\op}{\operatorname}
\newcommand{\supp}{\operatorname{supp}}

\newcommand{\ep}{\varepsilon}

\newcommand{\overbar}[1]{\mkern 1.5mu\overline{\mkern-1.5mu#1\mkern-1.5mu}\mkern 1.5mu}

\numberwithin{equation}{section}

\begin{document}

\title{A Lorentzian Lasry-Lions regularization theorem}
\author{{\sc Alec Metsch$^{1}$} \\[2ex]
      $^{1}$ Universit\"at zu K\"oln, Institut f\"ur Mathematik, Weyertal 86-90, \\
      D\,-\,50931 K\"oln, Germany \\
      email: ametsch@math.uni-koeln.de \\[1ex]
      Optimal transport, regularity of Kantorovich potentials, \\[1ex]weak KAM methods\\[1ex]
      }

\maketitle

\begin{abstract}
    \noindent
    The main goal of this paper is to establish a general Lorentzian Lasry-Lions regularization theorem: let $u$ be a function defined on a globally hyperbolic spacetime. Assume that its forward Lax–Oleinik evolution $Tu$ is locally semiconcave in a neighbourhood of $(t_0,y_0)$ and has only future-directed timelike superdifferentials at $(t_0,y_0)$. Then, for $t$ close to $t_0$ and $s>0$ sufficiently small, the function $\hat T_s\circ T_tu$ is of class $C_{\mathrm{loc}}^{1,1}$ in a neighbourhood of $y_0$. We provide sufficient conditions ensuring that the assumptions of the theorem are satisfied, and we present an application to optimal transport: under farily general assumptions, for any two intermediate measures along a displacement interpolation, there exists a $C^{1,1}_{loc}$-regular maximizing pair in the dual Kantorovich formulation. 
\end{abstract}

\tableofcontents

\section{Introduction}

If $H$ is a Hilbert space with induced norm $|\cdot|$ and $u:H\to \R$ is any function, the forward and backward Lax-Oleinik semigroups are defined as
\begin{align}
    T_tu(y):=\inf_{x\in H} \bigl(u(x)+\frac 1{2t} |x-y|^2\bigr), \quad \hat T_tu(x) := \sup_{y\in H} \bigl(u(y)-\frac 1{2t} |x-y|^2\bigr), \label{eqzr}
\end{align}
with $t>0$. If $u:H\to \R$ is bounded, Lasry and Lions \cite{Lasry/Lions} proved that for any $0<s<t$ the functions 
\begin{align}
    \hat T_s \circ T_t u \quad \text{ and } \quad T_s\circ \hat T_t u \label{eqzs}
\end{align}
belong to $C^{1,1}(H)$, and converge to $u$ as $t\to 0$ if $u$ is uniformly continuous. The result was motivated by the problem of finding smooth approximations in infinite dimensional spaces. See also \cite{Bernard2}.

Notice that $\frac{1}{2t}|x-y|^2$ is just the minimal action to go from $x$ to $y$ in time $t$:
\[
    \frac{1}{2t}|x-y|^2 = \inf\Bigl\{ \int_0^t \frac{|\dot \gamma_s|^2}{2}\, ds\mid \gamma:[0,t]\to H \text{ smooth },\ \gamma_0=x,\ \gamma_t = y\Bigr\}.
\]
For this reason, the semigroups \eqref{eqzr} make sense whenever $u$ is a function on a smooth manifold $M$ admitting a suitably well-behaved Lagrangian $L:TM\to \R\cup\{\infty\}$. Here, $\frac{1}{2t}|x-y|^2$ is replaced by the minimal action
\[
    h_t(x,y):=\inf\Bigl\{ \int_0^t L(\gamma_s,\dot \gamma_s)\, ds\mid \gamma:[0,t]\to M \text{ smooth },\ \gamma_0=x,\ \gamma_t = y\Bigr\}.
\]

A classical example is when $L$ is a Tonelli-Lagrangian on a complete Riemannian manifold $(M,g)$; roughly speaking, this means that $L$ is $C^2$ and enjoys several convexity and growth conditions. In this case, $h_t$ turns out to be real-valued and continuous \cite{Fathi/Figalli}. If $M$ is also compact, Bernard \cite{Bernard} was able to extend the result of Lasry-Lions: If $u:M\to \R$ is a bounded function, he proved that the functions in \eqref{eqzs} belong to $C^{1,1}(M)$ for $t>0$ and $s\leq s(t)$ sufficiently small. The motivation for Bernard’s result was to establish the existence of $C^{1,1}$ critical (viscosity) subsolutions to the stationary Hamilton–Jacobi equation, which is indeed a consequence of his result.

Bernard’s Lasry–Lions regularization was subsequently extended and generalized to the non-compact case in \cite{Cannarsa/Cheng/Fathi,Fathi/Figalli/Rifford,FathiHJ}: Indeed, Claim 4.7 in \cite{Cannarsa/Cheng/Fathi} (which relies on Theorem 1 of \cite{FathiHJ} and Lemma B.7 in \cite{Fathi/Figalli/Rifford}) states the following: If $V\Subset M$ is an open precompact set, $u:M\to \R\cup\{\pm\infty\}$ is any function and $t_0>0$ is such that $T_{t_0}(y_0)$ is finite at some point $y_0$, then $\hat T_s\circ T_tu$ belongs to $C_{loc}^{1,1}(V)$ for $t\in I\Subset (0,t_0)$ and $s>0$ sufficiently small.

In the special case $L(x,v)=\frac{|v|_g^2}{2}$, it was shown in \cite{Cannarsa/Cheng/Fathi} that the $C_{loc}^{1,1}$-property of the functions $\hat T_s\circ T_tu$ can be used to study topological properties of both the cut locus and the set of non-differentiability points of the Riemannian distance function – it was shown that both sets are locally contractible.

Motivated by these results, the main goal of this paper is to establish a Lorentzian analogue of the Lasry–Lions-type regularization theorem in the setting of a globally hyperbolic spacetime. We study the Lax-Oleinik semigroups w.r.t.\ the standard Lorentzian Lagrangian \eqref{eqzw} which gives rise – via the minimal time-$1$-action – to the standard cost function  (some version of the Lorentzian time separation (\eqref{eqzx} with $t=1$)) in the Lorentzian optimal transport problem \eqref{eqzy}. After proving the Lasry–Lions-type regularity result, we can therefore provide an application to optimal transport: under suitable assumptions, we show that the dual Kantorovich problem for intermediate times along a displacement interpolation (cf. Definition \ref{defa}) admits a solution given by $C^{1,1}_{loc}$-functions (the analogous result also holds in the Riemannian framework, although I am not aware of any reference).

It is worth mentioning that our main theorem yields the same topological properties for the future causal cut locus and the singularities of the Lorentzian distance function as in the Riemannian case. This result, however, was already established in an earlier work \cite{Metsch3}, where a corresponding special case of our main theorem was treated.

\subsection{Setting and results}

To state our precise results, let $(M,g)$ be a globally hyperbolic spacetime, and fix $0<p< 1$.

We consider the Lagrangian $L:TM\to \R\cup\{+\infty\}$ by
\begin{align}
    L(x,v):=
    \begin{cases}
        -|v|^{p}_g, & \text{ if $v$ is future causal},
        \\[10pt]
        +\infty,& \text{otherwise,}
    \end{cases} \label{eqzw}
\end{align}
where $|v|_g:=\sqrt{|g(v,v)|}$.
For $t>0$, the Lagrangian gives rise to the minimal time-$t$-action $c_t$ via
\begin{align*}
	c_t(x,y) :=\inf\Bigl\{\int_0^t L(\gamma_s,\dot \gamma_s)\, ds\mid \gamma:[0,t] \to M\text{ is piecewise smooth, } \gamma_0=x, \gamma_t= y\Bigr\}.
\end{align*}
If $\ell$ denotes the time separation function \eqref{eqzz}, it is not difficult to show that
\begin{align}
	c_t(x,y)=
	\begin{cases}
	-t^{1-p} \ell(x,y)^{p}, &\text{ if } t>0 \text{ and }y\in J^+(x),
	\\[10pt]
	+\infty, & \text{ otherwise.} \label{eqzx}
	\end{cases}
\end{align}
where $J^+(x)$ denotes the causal future of $x$ (note that, with our convention, $0$ is a future causal vector). As in the Riemannian case, we can define the forward and backward Lax-Oleinik/Hopf-Lax semigroups $Tu,\hat Tu:(0,\infty)\times M\to \overbar \R:=\R\cup\{\pm \infty\}$ of a function $u:M\to \overbar \R$ via
\begin{align*}
    T_tu(y):=\inf_{x\in J^-(y)} (u(x)+c_t(x,y))
    \quad \text{ and } \quad
    \hat T_tu(x) := \sup_{y\in J^+(x)} (u(y)-c_t(x,y)).
\end{align*}
We can now state our main theorem. Recall that a function defined on a smooth manifold is called locally semiconcave if it is so in local coordinates, see Definition \ref{semiconvexity}. The set of super-differentials of a function $u$ at $x$ is denoted by $\partial^+u(x)$, and $\C^*_y\subseteq T_yM^*$ denotes the dual future causal cone (cf.\ Definition \ref{causalco}).

\begin{theorem}[Lasry-Lions regularization]\label{main}
    Let $N$ be a smooth manifold, and let $u:N\times M\to \R\cup\{\pm\infty\},\ (x,y)\mapsto u(x,y)=u_x(y)$, be such that the following family of Lax-Oleinik functions
    \begin{align*}
        \hat u:(0,\infty)\times N\times M\to \R\cup\{\pm \infty\},\ (t,x,y)\mapsto (T_tu_x)(y), 
    \end{align*}
    is locally semiconcave on an open neighbourhood of $(t_0,x_0,y_0)$ with
    \begin{align}
        \partial^+(T_{t_0}u_{x_0})(y_0)\Subset \op{int}(\C_{y_0}^*). \label{eqzv}
    \end{align}
    Then there exist three open neighbourhoods $I$, $U$ and $V$ of $t_0$, $x_0$ and $y_0$, respectively, and some number $s_0>0$ such that the following properties hold:
    \begin{enumerate}[(i)]

    \item The functions $\hat T_s\circ T_tu_x$, $(s,t,x)\in (0,s_0]\times I\times U$, are $C^{1,1}_{loc}$ on $V$.
    \item For $(s,t,x,y)\in (0,s_0]\times I\times U\times V$, there exists a unique $z\in M$ with $\hat T_s\circ T_tu_x(y)=T_tu_x(z)-c_s(y,z)$, and necessarily $\ell(y,z)>0$. Moreover, setting $z=y$ if $s=0$, $z$ depends continuously on $(s,t,x,y)$.
    \end{enumerate}
\end{theorem}
A particular case is when $u_x=u$ does not depend on $x$. The dependence on $x$, however, becomes relevant when studying the topological properties of the time separation function \cite{Metsch3}, à la \cite{Cannarsa/Cheng/Fathi}; see Section \ref{discussion} for a discussion of this phenomenon and why it does not arise in the Riemannian case. We also refer to Section \ref{discussion} for a discussion of hypotheses ensuring local semiconcavity of $\hat u$ and \eqref{eqzv}, and why these two properties have no Riemannian counterpart. Notice that an analogous version of Theorem \ref{main} holds when the roles of the forward and backward semigroups are interchanged.
\medskip

Our second main result is an application of the theorem to the optimal transport problem. Before stating the precise result, we introduce the necessary background and basic concepts.

We consider the optimal transport problem on $M$ associated with the cost function $c_1$.
\begin{align}
    C(\mu,\nu)=\inf_{\pi\in \Pi_{\leq}(\mu,\nu)}\int_{M\times M} c_1(x,y)\, d\pi(x,y)\in [-\infty,0]\cup \{\infty\}, \label{eqzy}
\end{align}
where $\mu,\nu\in {\cal P}:={\cal P}(M)$ are Borel probability measures on $M$,  $\Pi_{\leq}(\mu,\nu)$ denotes the set of causal couplings of $\mu$ and $\nu$ (i.e.\ all Borel probability measures $\pi$ on $M\times M$ having $\mu$ and $\nu$ as marginals and such that $\pi(J^+)=1$). We adopt the convention that $\inf(\emptyset):=+\infty$.
This is the standard cost function for the Lorentzian optimal transport problem and it has been studied extensively (see e.g.\ \cite{Kell, McCann, Mondino/Suhr, Suhr, Braun} for the smooth and \cite{Braun, Mondino/Cavalletti, Octet} for the synthetic setting). Notice that \eqref{eqzy} is often stated as a maximization problem, cf.\ Section \ref{discussion}.

In the following fundamental definition, $\op{CGeo([0,1],M)}$ denotes the Polish space of maximizing future causal geodesics over $[0,1]$ (cf.\ Remark \ref{remf}), $e_t:\op{CGeo([0,1],M)}\to M$ the evaluation map at time $t$, and $_\#$ is used to denote the push-forward of a measure.

\begin{definition}[Dynamical optimal coupling]\label{defa}\rm
    We say that $\Pi\in \mathcal{P}(\op{CGeo([0,1],M)})$ is a \emph{dynamical optimal coupling} from $\mu_0:=(e_0)_\#\Pi$ to $\mu_1:=(e_1)_\#\Pi$ if $C(\mu_0,\mu_1)$ is finite and
    $(e_0,e_1)_\#\Pi\in \Pi_{\leq}(\mu_0,\mu_1)$ is an optimal coupling. $(\mu_t)$ is called \emph{displacement interpolation}.
\end{definition}

For a comparison to the notion of $p$-geodesics, as introduced by McCann \cite{McCann2}, see Section \ref{discussion}.
Using standard gluing methods, it is not difficult to prove that  $\pi_{s,t}:=(e_s,e_t)_\#\Pi$ constitutes an optimal coupling of $\mu_s$ and $\mu_t$ for all $0\leq s\leq t\leq 1$ (Remark 2.32 in \cite{Mondino/Cavalletti}).

\begin{definition}[Calibrated pair]\rm\label{def1231}
Let $\mu,\nu\in {\cal P}(M)$ such that $C(\mu,\nu)$ is finite, and suppose that $\pi\in \Pi_{\leq}(\mu,\nu)$ is an optimal coupling. We say that a pair of functions $\varphi,\psi:M\to \R\cup\{\pm\infty\}$ is \emph{$(c_1,\pi)$-calibrated} if 
    \begin{enumerate}[(a)]
        \item it is a $c_1$-\emph{subsolution}, i.e.\ $\psi(y)-\varphi(x) \leq c_1(x,y)$ for all $x,y\in M$ and
        \item $\psi(y)-\varphi(x) =c_1(x,y)\ (=-\ell(x,y)^p)\quad   \pi\text{-a.e.}$.
    \end{enumerate}
    Here, we use the convention $\pm \infty \mp \infty := -\infty$.
\end{definition}

\begin{theorem}[$C^{1,1}_{loc}$-optimal pairs]\label{thmg}
    Let $\mu_0,\mu_1\in {\cal P}(M)$ be such that $C(\mu_0,\mu_1)$ is finite, and let $\Pi$ be a dynamical optimal coupling with associated optimal coupling $\pi:=(e_0,e_1)_\#\Pi$. Let $(\varphi,\psi):M\to \overbar \R$ be a $(c_1,\pi)$-calibrated pair, and $0\leq s<t\leq 1$. Suppose that
    \begin{enumerate}[(i)]
        \item $\supp(\mu_0)$ or $\supp(\mu_1)$ is causally compact, i.e.\ 
        for all $K\subseteq M$ compact the intersections $J^\pm(\supp(\mu_i))\cap K$ are compact.
        \item $\Pi$ is concentrated on timelike geodesics.
    \end{enumerate}
    Then there exists two open sets $\Omega_s,\Omega_t\subseteq M$ with $\mu_s(\Omega_s)=1$, $\mu_t(\Omega_t)=1$, and a $(c_1,\pi_{s,t})$-calibrated pair $(\varphi_s,\psi_t)$ such that $\varphi_s\in C^{1,1}_{loc}(\Omega_s)$ and $\psi_t\in C^{1,1}_{loc}(\Omega_t)$.
    Moreover, if $(\varphi,\psi)$ is optimal in the dual Kantorovich formulation, so is $(\varphi_s,\psi_t)$.
\end{theorem}

\subsection{Discussion}\label{discussion}
(1) As already mentioned, the dependence on $x$ in Theorem \ref{main} is essential in order for the statement to imply topological properties of the singularities of the time separation function. Indeed, one considers 
\[
    u_x = \chi_x:M\to \R\cup\{\infty\},\ \chi_x(y):=
    \begin{cases}
        0,& \text{ if } y=x,
        \\[5pt]
        +\infty,& \text{ otherwise.}
    \end{cases}
\]
 To derive the corresponding topological properties, it is necessary, for instance, that $s_0$ does not depend on $x$, at least locally. This dependence is not needed in the Riemannian case, due to an argument based on considering a suitable function $u$ on the product manifold, which is again Riemannian, whereas the product of two Lorentzian manifolds is not Lorentzian.
 \\

\noindent(2) Of course, an important special case in Theorem \ref{main} is when $u_x\equiv u$ does not depend on $x$, so that $\hat u=\hat u(t,y)$. In contrast to the Riemannian setting, we require the local semiconcavity of $\hat u$ near $(t_0,y_0)$. Although this property also plays a crucial role in the Riemannian case, it follows there from the finiteness of $T_{t_1}u(y_1)$ for some $t_1>t_0$ and $y_1\in M$, see Theorem 1 in \cite{FathiHJ}; the theorem fails in the Lorentzian setting, which is why we impose this assumption. In fact, I am not aware of any version of the theorem except for Theorem \ref{m} in the present paper. Finally, Assumption \eqref{eqzv} has no counterpart in the Riemannian case, but it is used in several technical parts of the proof. 
\medskip

Theorem \ref{m} states that, if $N=M$, $u_x\equiv u$ is $+\infty$ outside of a causally compact set and $x,y\in M$ are such that $u(x)$ is finite with $y\in I^+(x)$ and
\[
    T_1u(y) = u(x) + c_1(x,y),
\]
then the hypothesis of Theorem \ref{main} are satisfied at any point $(t,\gamma_t)$, $t\in (0,1)$, where $\gamma:[0,1]\to M$ is a maximizing geodesic from $x$ to $y$.

If $(\varphi,\psi)$ is a $(c_1,\pi)$-calibrated pair as in Theorem \ref{thmg}, then we can set $u:=\varphi$. Notice that $\Pi$ is concentrated on the set of future timelike geodesics such that $\psi(\gamma_1)-\varphi(\gamma_0)=c_1(\gamma_0,\gamma_1)$. This gives
\[
    T_1\varphi(\gamma_1) = \varphi(\gamma_0)+c_1(\gamma_0,\gamma_1)
\]
and we may thus apply the above criterion to conclude that $\hat T_\tau \circ T_{t+\tau}\varphi$ is $C^{1,1}_{loc}$ near $\gamma_t$ for small $\tau>0$. An analogous result holds with $s$ replacing $t$. Using a partition of unity argument, Theorem \ref{thmg} is then a consequence of the fact that $(t-s)^{p-1}(\hat T_\tau\circ T_{s+\tau}\varphi,\hat T_\tau \circ T_{t+\tau}\varphi)$ is $(c_1,\pi_{s,t})$-calibrated.
\medskip

\noindent (3) The optimal transport problem we study is usually defined as a maximazation problem \cite{McCann, Mondino/Cavalletti, Octet}: For two probability measure $\mu,\nu\in \mathcal{P}(M)$, one defines $\ell_p(\mu,\nu)=-\infty$ if $\Pi_{\leq}(\mu,\nu)=\emptyset$ and otherwise as
\begin{align*}
    \ell_p(\mu,\nu) := \bigg(\sup_{\pi\in \Pi_{\leq}(\mu,\nu)} \int_{M\times M} \ell(x,y)^p\, d\pi(x,y)\bigg)^\frac{1}{p}.
\end{align*}
Up to taken the $p$-th root, this definition differs from ours only be a sign ($\ell_p^p(\mu,\nu)=-C(\mu,\nu)$), so the two problems are equivalent. The advantage of the standard formulation is that $\ell_p$ satisfies the reverse triangle inequality (it is a time separation function on $\mathcal{P}(M)$ in the sense of \cite{Octet}); our definition, on the other hand, focuses on the Lagrangian formulation.

Let us also compare our notion of displacement interpolation (used by \cite{Mondino/Cavalletti}) to the notion of $p$-geodesics introduced by McCann \cite{McCann2}. A $p$-geodesic is a narrowly continuous curve $(\mu_t)$ satisfying 
\[
    \ell_p(\mu_s,\mu_t) = (t-s)\ell_p(\mu_0,\mu_1)\in (0,\infty)
\]
for all $0\leq s<t\leq 1$. Any displacement interpolation with non-zero total cost is a $p$-geodesic \cite{Mondino/Cavalletti}. Conversely, if $(\mu_t)$ is a $p$-geodesic, $\mu_0$ and $\mu_1$ have compact support and each optimal coupling is concentrated on $I^+$, then one can show by standard arguments that $(\mu_t)$ is induced by a dynamical optimal coupling. In particular, the two definitions coincide for the important class of $p$-separated measures, as defined in \cite{McCann2}.

\subsection{Plan of the paper}

The structure of this paper is as follows: 
In Section 2 we describe the precise setting considered in this paper and collect the necessary background material. For instance, we study elementary properties of the Lagrangian, its Hamiltonian and maximizing curves. In Section 3, we study in great detail a local version of the Lax-Oleinik semigroup. In Section 4, we give the proof of the main theorem and give sufficient conditions ensuring the validity of the assumptions. Finally, we prove our optimal transport result in Section 5.

\section{Preliminaries}\label{sec3}

\textbf{For the rest of this paper}, let $p\in (0,1)$ and let $(M,g)$ be an $n$-dimensional globally hyperbolic spacetime, where the metric $g$ is taken to have signature $(-,+,...,+)$. In particular, $M$ is time-oriented. We understand $0$ to be a future causal vector. A curve is always assumed to be piecewise smooth if not otherwise said. In particular, a curve is referred to as future causal (timelike) if it is piecewise smooth and future causal (timelike). We denote by $I^+$ resp.\ $J^+$ the set of pairs $(x,y)$ which can be connected by a future timelike resp.\ future causal curve. The Lorentzian distance function is denoted by $\ell$:
\begin{align}
    \ell(x,y):=
    \begin{cases}
        \sup_\gamma {\ell}_g(\gamma),&\text{ if } (x,y)\in J^+,
        \\
        0,& \text{ else,}
    \end{cases} \label{eqzz}
\end{align}
where the supremum is taken over all future causal curves $\gamma:[a,b]\to M$ connecting $x$ to $y$ and the (Lorentzian) length ${\ell}_g$ of a future causal curve $\gamma:[a,b]\to M$ is defined as
\[
    \int_a^b \sqrt{|g(\dot \gamma_t,\dot \gamma_t)|}\, dt.
\]
For $x\in M$, we denote by $\C_x\subseteq T_xM$ the cone of future causal vectors. Note that $\C_x$ is closed. We also set $\C:=\{(x,v)\in TM\mid v\in \C_x\}$. We can equip $M$ with a complete Riemannian metric, which will be fixed and denoted by $h$. The $h$- and $g$-norm of a tangent vector $v\in T_xM$ are denoted by $|v|_h$ and $|v|_g:=\sqrt{|g(v,v)|}$, respectively. The Riemannian distance is denoted by $d_h$.

\subsection{The Lagrangian}
In this short section, we define the Lagrangian and the associated minimal action. We recall the most important definitions/notations and results that will be used many times in this paper. Some standard and well-known proofs are deferred to the appendix.

\begin{definition}\rm
    We define the Lagrangian $L:TM\to \R\cup\{+\infty\}$ as
    \begin{align*}
        L(x,v):=
        \begin{cases}
          -|v|^{p}_g, & \text{ if } v\in \C_x,
          \\[10pt]
          +\infty,& \text{otherwise.}
        \end{cases}
    \end{align*}
    See also \cite{Mondino/Suhr}, Section 2, and \cite{McCann2}, Section 3.
\end{definition}

\begin{definition}\rm
\begin{enumerate}[(a)]
    \item 
    The \emph{action} of a curve $\gamma:[a,b]\to M$ is defined by
    \begin{align*}
        \mathbb{L}(\gamma):=\int_a^b L(\gamma_t,\dot \gamma_t)\, dt\in (-\infty,\infty].
    \end{align*}
    Note that $\mathbb{L}(\gamma)$ is finite if and only if $\gamma$ is future causal.
    \item A curve $\gamma:[a,b]\to M$ is called an $L$-\emph{minimizer} if for any other curve $\tilde \gamma:[a,b]\to M$ with the same endpoints, we have $\mathbb{L}(\gamma)\leq \mathbb{L}(\tilde \gamma)$.
\end{enumerate}
\end{definition}

\begin{notation}[On the terminology] \rm\label{remf}
    We reserve the term \emph{maximize} to refer specifically to the Lorentzian length functional. That is, a maximizing curve $\gamma:[a,b]\to M$ is a future causal curve that satisfies $\ell_g(\gamma)=\ell(\gamma_a,\gamma_b)$.
\end{notation}

 The following result is well-known. 

\begin{theorem}\label{asdfghj}
    The set $J^+$ is closed. For any two points $(x,y)\in J^+$, there exists a maximizing geodesic connecting $x$ to $y$. Moreover, any maximizing curve must be a pregeodesic.
\end{theorem}
\begin{proof}
    See \cite{ONeill}, Chapter 14, Proposition 19 and \cite{Minguzzi}, Theorem 2.9.
\end{proof}

\begin{lemma}[Minimizers and geodesics]\label{lt}
	Let $x\in M$, $y\in J^+(x)$ and $a<b$. Let $\gamma:[a,b]\to M$ be a curve connecting $x$ to $y$. Then:
	\begin{align*}
	\gamma \text{  is a maximizing geodesic} \Rightarrow \gamma \text{ is $L$-minimizing }  
	\end{align*}
	If $y\in I^+(x)$, the implication becomes an equivalence. 
\end{lemma}
\begin{proof}
	See Appendix \ref{proofs}.
\end{proof}

\begin{remark}[Euler-Lagrange flow = geodesic flow]\rm
	It is not difficult to verify that the second derivative along the fibres, $\frac{\partial^2 L}{\partial v^2}$, is positive definite at every point $(x,v)\in \op{int}(\C)$ (\cite{Mondino/Suhr}, Lemma 2.1). As a consequence, it is well-known (\cite{Fathi}, Section 2.4) that there exists a smooth local flow (the Euler-Lagrange flow) $\phi_t$ on $\op{int}(\C)$ whose orbits are precisely the speed curves of extremals for $L$; that is, the curves of the form $(\gamma_t,\dot \gamma_t)$, where $\gamma:I\to M$ is a $C^2$-curve satisfying, in local coordinates, the Euler-Lagrange equation
	\begin{align*}
	\frac{d}{dt}\bigg(\frac{\partial L}{\partial v}(\gamma_t,\dot \gamma_t)\bigg)=\frac{\partial L}{\partial x}(\gamma_t,\dot \gamma_t).
	\end{align*}
	Moreover, every future timelike $L$-minimizing curve $\gamma:[a,b]\to M$ solves the Euler-Lagrange equation. Since the future timelike $L$-minimizing curves are exactly the future timelike maximizing geodesics, and since every future timelike geodesic is locally maximizing (Proposition 7.3 in \cite{Metsch1}), it follows that the Euler-Lagrange flow coincides with the geodesic flow restricted to the invariant set $\op{int}(\C)$.
\end{remark}

\begin{definition}[Minimal action]\rm\label{minac}
	For $t>0$, the \emph{minimal time-$t$-action} to go from $x$ to $y$ is defined by
	\begin{align*}
	\C(t,x,y):=c_t(x,y) :=\inf\{\mathbb{L}(\gamma)\mid \gamma:[0,t] \to M\text{ is a curve connecting } x \text{ to } y\}.
	\end{align*}
\end{definition}

\begin{corollary}[Formula for the minimal action]\label{a}
	We have the following identity:
	\begin{align*}
	c_t(x,y)=
	\begin{cases}
	-t^{1-p} \ell(x,y)^{p}, &\text{ if } y\in J^+(x),
	\\[10pt]
	+\infty, & \text{ otherwise.}
	\end{cases}
	\end{align*}
	Moreover, for any $x\in M$, $y\in J^+(x)$ and $t>0$, there exists a smooth $L$-minimizing geodesic $\gamma:[0,t]\to M$ connecting $x$ to $y$.
\end{corollary}
\begin{proof}
	This follows immediately from Theorem \ref{asdfghj} and Lemma \ref{lt}.
\end{proof}

\begin{definition}[dual cone]\rm\label{causalco}
	For each $x\in M$, we denote the canonical isomorphism by
	\begin{align*}
	T_xM\to T_x^*M,\ v\mapsto v^\flat:=g(v,\cdot).
	\end{align*}
	We define $\C_x^*$ as the image of $\C_x$ under this isomorphism, and $\C^*:=\{(x,p)\in T^*M\mid p\in \C_x^*\}.$
\end{definition}

\begin{definition}[Semiconcavity and super-differentials]\rm \label{semiconvexity}
\begin{enumerate}[(i)]
    \item A function $f:U\to \R$ defined on the open set $U\subseteq \R^n$ is called \emph{semiconcave} if there exists $C\in \R$ such that
    \begin{align}
        \forall x\in U,\ \exists p\in \R^n:\ \forall y\in \R^n: f(y)\leq f(x)+\langle p,y-x\rangle -C|y-x|^2. \label{eqzad}
    \end{align}
    We shall say that $f$ is \emph{$C$-concave}. The notion of \emph{local semiconcavity} is defined in an obvious way. A family of functions $f_i:U\to \R$ is called \emph{uniformly semiconcave} if $C$ in \eqref{eqzad} can be chosen independently of $i$.
    \item A function $f:U\to \R$ defined on the open subset $U\subseteq M$ is called \emph{locally semiconcave} if it is so when computed in some (hence any) chart.
    A family of functions $f_i:U\to \R$ is called \emph{uniformly locally semiconcave} if, for each $x\in U$, there exists a chart $\varphi$ around $x$ such that the family of functions $(f_i\circ \varphi^{-1})_i$ is uniformly semiconcave.
    \item A function $f:U\to \overbar \R$ defined on the open subset $U\subseteq M$ is said to be \emph{super-differentiable} at $x$ with \emph{super-differential} $p\in T_x^*M$ if $f(x)\in \R$ and 
    \[
        f(\exp^h_x) \leq f(x) +p(v)+o(|v|_h),
    \]
    where $\exp_x^h$ denotes the exponential function of the Riemannian metric $h$.
    We denote by $\partial^+f(x)$ the set of all super-differentials.
    \item The notion of \emph{semiconvexity/sub-differentiability} is defined analogously by replacing "$\leq$" with "$\geq$", with $\partial^-f$ denoting the set of subdifferentials.
\end{enumerate}
    
\end{definition}

\begin{theorem}[Semiconcavity and supergradients of the minimal action]\label{A} Let $\C:(0,\infty)\times M\times M\to \R\cup\{+\infty\}$ be the function from Definition \ref{minac}.
	\begin{enumerate}[(a)]
		\item $\C$ is real-valued and continuous on $(0,\infty)\times J^+$, and even locally semiconcave on $(0,\infty)\times I^+$.
		\item If $x\in M$, $y\in I^+(x)$ and $t>0$, then the set of super-differentials of $\C$ at the point $(t,x,y)$ is given by
		\begin{align*}
		\partial^+ \C(t,x,y)
		=
		\op{conv}\left(\bigg\{\left((p-1)\Bigl(\frac{\ell(x,y)}{t}\Bigr)^p,-\frac{\partial L}{\partial v}(x,\dot \gamma_0), \frac{\partial L}{\partial v}(y,\dot \gamma_t)\right)\bigg\}\right),
		\end{align*}
		where the set is taken over all maximizing geodesics $\gamma:[0,t]\to M$ connecting $x$ to $y$.
		
		In particular, $\C$ is differentiable at $(t,x,y)$ if and only if there is a unique maximizing geodesic connecting $x$ to $y$ in time $t$ (equivalently, in time $1$).
	\end{enumerate}
\end{theorem}
\begin{proof}
	See \cite{McCann2}, Proposition 3.4 for (a). A complete proof in the case $p=\frac 12$ can be found in \cite{Metsch3}. The proof carries over to $p\in (0,1]$ with no additional difficulties.
\end{proof}

\begin{definition}[Legendre transform]\rm\label{Legendre}
	The \emph{Legendre transform} of $L$ is the map
	\begin{align*}
	{\cal L}:\op{int}(\C)\to T^*M,\ 
	(x,v)\mapsto \left(x,\frac{\partial L}{\partial v}(x,v)\right).
	\end{align*}
\end{definition}

\begin{proposition}[Legendre transform is a diffeomorphism]\label{pro2}
	The Legendre transform is a diffeomorphism onto its image
	$\op{im}({\cal L})=\op{int}(\C^{*})$.
\end{proposition}
\begin{proof}
	We have
	\begin{align}
	\bigg(x,\frac{\partial L}{\partial v}(x,v)\bigg)
	=
	\bigg(x,p |v|_g^{{p-2}}v^\flat\bigg),
	\label{eqk}
	\end{align}
	which is clearly a diffeomorphism from $\op{int}(\C)$ to $\op{int}(\C^*)$ because $p\in (0,1)$.
\end{proof}

\begin{definition}[Hamiltonian]\rm
	The \emph{Hamiltonian} associated with $L$ is the function $H:T^*M\to \R\cup\{+\infty\}$ defined by
	\begin{align*}
	H(x,p):=\sup\{pv-L(x,v)\mid v\in T_xM\}.
	\end{align*}
\end{definition}

\begin{lemma}[See \cite{Mondino/Suhr}, Section 2, and \cite{McCann2}, Lemma 3.1]\label{lba}
	We have
	\begin{align*}
	H({\cal L}(x,v))=\frac{\partial L}{\partial v}(x,v)(v)-L(x,v)
	=(p-1) L(x,v)
	\end{align*}
	for all $(x,v)\in \op{int}(\C)$. 
\end{lemma}
\begin{proof}
    See Appendix \ref{proofs}.
\end{proof}

\begin{remark}[The Hamiltonian and the geodesic flow]\rm \label{reme}
	From the above lemma, we conclude that $H$ is smooth on the open set $\op{int}(\C^{*})$, and satisfies the identity $H({\cal L}(x,v))=\frac{\partial L}{\partial v}(x,v)(v)-L(x,v)$. It is therefore well-known (\cite{Fathi}, Proof of Theorem 2.6.4) that the Euler-Lagrange flow $\phi_t$ (i.e.\ the geodesic flow on $\op{int}(\C)$) is conjugate, via the diffeomorphism ${\cal L}:\op{int}(\C)\to \op{int}(\C^*)$, to the Hamiltonian flow $\psi_t$ on $\op{int}(\C^{*})$. The latter is understood with respect to its canonical symplectic structure as an open subset of the cotangent bundle.
\end{remark}

\subsection{The Lax-Oleinik semigroup}

In this subsection, we define the (Lorentzian) Lax-Oleinik/Hopf-Lax semigroup, and we study elemantary properties.

\begin{definition}[Lax-Oleinik semigroup]\rm\label{def2}
	The \emph{forward Lax-Oleinik semigroup} is the family of maps $(T_t)_{ t> 0}$, defined on the space of functions $u:M\to \overbar \R$ by
	\begin{align*}
	T_tu:M\to \overbar \R,\ T_tu(y):=\inf\{u(x)+c_t(x,y)\mid x\in M\},
	\end{align*}
	where we use the convention $-\infty+\infty:=+\infty$.
	The \emph{backward Lax-Oleinik semigroup} is the family of maps $(\hat T_t)_{t> 0}$, defined on the space of functions $u:M\to \overbar \R$ by
	\begin{align*}
	\hat T_tu:M\to \overbar \R,\ \hat T_tu(x):=\sup\{u(y)-c_t(x,y)\mid y\in M\},
	\end{align*}
	where we use the convention $\infty-\infty:=-\infty$.
	If $u:M\to \overbar \R$ is any function, we also write
	\begin{align*}
	Tu:&\ (0,\infty)\times M\to \overbar \R, \ (t,y)\mapsto T_tu(y),
	\end{align*}
	and we use the analogous notation $\hat Tu$.

    For convenience, we also set $T_0u=\hat T_0u:=u$.
\end{definition}

\begin{remark}[About the convention]\rm
    To stay away from computations involving $\pm \infty \mp \infty$, let us mention that the definition of the Lax-Oleinik semigroup can equivalently be given as in the introduction:
    \begin{align*}
	T_tu:M\to \overbar \R,\ T_tu(y):=\inf\{u(x)+c_t(x,y)\mid x\in J^-(y)\}.
	\end{align*}
    In this case, $c_t(x,y)$ is always finite.
\end{remark}

\begin{lemma}[Semigroup property and composition]\label{semigroup}
	\begin{enumerate}[(i)]
		\item $(T_t)$ and $(\hat T_t)$ are semigroups, that is, $T_{t+s}=T_t\circ T_s$ and $\hat T_{t+s}=\hat T_t\circ \hat T_s$ for all $t,s\geq 0$.
		\item It holds $\hat T_t\circ T_t u\leq u$ and $T_t\circ \hat T_t u \geq u$ for all $t\geq 0$ and $u:M\to \overbar \R$.
	\end{enumerate}
\end{lemma}
\begin{proof}
	(ii) is trivial from the definition (cf.\ \cite{Ambrosio/Gigli}), so we only prove (i) for the forward semigroup. The statement is trivial if $t=0$ or $s=0$. Thus, fix $t,s> 0$ and $y\in M$.
    
    For $\ep>0$, let $x\in J^-(y)$ be such that
    \begin{align}
        T_{t+s}u(y) +\ep \geq u(x)+c_{t+s}(x,y). \label{eqzo}
    \end{align}
    If $\gamma:[0,t+s]\to M$ denotes a maximizing geodesic from $x$ to $y$, it follows that
    \[
        u(x)+c_{t+s}(x,y) =
        u(x) +c_s(x,\gamma_s)+c_t(\gamma_s,y)
        \geq 
        T_su(\gamma_s) +c_t(\gamma_s,y)
        \geq T_t \circ T_su(y).
    \]
    Combining with \eqref{eqzo}, this shows $T_t\circ T_s \leq T_{t+s}$ since $\ep$ was arbitrary. Conversely, for $\ep>0$, fix $x\in J^-(y)$ with
    \[
        T_{t}\circ T_su(y) +\ep \geq T_su(x)+c_{t}(x,y),
    \]
    and let $z\in J^-(x)$ with
    \[
        T_su(x)+\ep \geq u(z) +c_s(z,x).
    \]
    Then it follows easily from the definition of the minimal action that
    \[
        T_t \circ T_su(y) +2\ep \geq u(z)+c_t(z,x)+c_s(x,y) \geq u(z)+c_{t+s}(z,y) \geq T_{t+s}u(y),
    \]
    which proves the other implication since $\ep$ was arbitrary.
\end{proof}

\section{A local Lax-Oleinik semigroup}

In order to prove the main result, it is convenient to study some local type of the Lax-Oleinik semigroup. This method has already been used in the case of Tonelli Lagrangians systems in \cite{Fathi/Figalli/Rifford,Cannarsa/Cheng/Fathi}. Although the choice of $V_0$ in the following definition is quite arbitrary, the proof of the main theorem consists in showing that the backward semigroup $\hat T_su_x$ locally conincides with the local one. Hence, all properties we are able to prove for the local version easily extend to the global version.

Let us first describe the general setting of functions that we are studying in this section.

\begin{definition}[Local Lax-Oleinik semigroup]\rm
    Let $s> 0$, $V_0\subseteq M$ be an open set, and $f:V_0\to \R$ be a function. We define the \emph{local backward Lax-Oleinik semigroup} of $f$ as
    \begin{align}
        \hat T_s^{V_0}f:V_0\to \R\cup\{\infty\},\  \hat T_s^{V_0}f(y):=\sup\{f(z)-c_s(y,z)\mid z\in V_0\}. \label{eqzaa}
    \end{align}
\end{definition}

\begin{remark}\rm
    We use the variable $s$ instead of $t$. The reason behind this is consistency with Theorem \ref{main}: The functions $f$ play the role of the functions $T_tu_x$ of Theorem \ref{main}.

    Of course, one could also consider functions $f:V_0\to \overbar \R$ using the same conventions as in Definition \ref{def2}. Since the local semigroup is merely a technical tool, we assume for simplicity that $f$ is real-valued.
\end{remark}

\begin{definition}[Causally convex sets]
    A set $V\subseteq M$ is \emph{causally convex} if all future causal curves that start and end in $V$ remain within $V$ for all times.
\end{definition}

\begin{remark}[The general framework]\rm\label{rmzd}
    If $K_0\subseteq \op{int}(\C^*)$ is a compact set and $V_0\Subset M$ is a precompact open and causally convex set, denote by $\mathcal{F}=\mathcal{F}(K_0,V_0)$ the family of locally semiconcave functions $f:V_0\to \R$ satisfying:
    \begin{align}
        \{(y,p)\in T^*M\mid  y\in V_0,\ f\in \mathcal{F}, p\in \partial^+f(y)\}\Subset K_0. \label{eqzd}
    \end{align}
    In particular, for any future causal curve $\gamma:[a,b]\to V_0$, the composition $f\circ \gamma$ is Lipschitz (hence a.e.\ differentiable) with
    \begin{align*}
        \frac{d}{dt}(f\circ \gamma)(t)\leq -C_0\, |\dot \gamma_t|_h
    \end{align*}
    for a.e.\ $t$ and some constant $C_0=C_0(K_0,V_0)$ (recall our convention of $g=(-,+,...,+)$). 
\end{remark}

\subsection{Elementary properties and semiconvexity}

In this section we study elementary properties of the local Lax-Oleinik semigroup for functions in the set $\mathcal{F}(K_0,V_0)$, and we then restrict our attention on the behaviour on some open set $V\Subset V_0$ and for sufficiently small $s$. We prove existence and uniqueness of optimal points $z$ in \eqref{eqzaa} under suitable assumptions.  We conclude this section proving local semiconvexity of the local Lax-Oleinik semigroup.

\begin{lemma}[Locus and uniqueness of optimal points]\label{lbb}
    Let $K_0,V_0$ and $\mathcal{F}$ be as in Remark \ref{rmzd}, and let $f\in \mathcal{F}$, $y\in V_0$ and $s>0$. Then the following properties hold:
    \begin{enumerate}[(i)]
    \item 
    It holds
    \begin{align}
        \hat T_s^{V_0}f(y)>f(y). \label{eqac}
    \end{align}
    
    \item If $z\in V_0$ is optimal for $\hat T_s^{V_0}f(y)$, then $z\in I^+(y)$. 
    
    \item If moreover $\hat T_s^{V_0}f$ is differentiable at $y$, then there exists at most one optimal point $z$ and it satisfies
    \begin{align*}
        d_y(\hat T^{V_0}_sf)=\frac{\partial L}{\partial v}(y,\dot \gamma_0) \quad \text{ and } \quad \frac{\partial L}{\partial v}(z,\dot \gamma_s)\in \partial^+f(z),
    \end{align*}
    where $\gamma:[0,s]\to M$ is the unique(!) maximizing geodesic connecting $y$ to $z$.   
    \end{enumerate}
\end{lemma}
\begin{proof}
    \begin{enumerate}[(i)]
    \item 
     Let $\gamma:[0,\ep_0]\to M$ be a future timelike geodesic with $\gamma_0=y$.  
    Since $f$ is locally semiconcave (hence locally Lipschitz) on some neighbourhood of $y$, there is $C>0$ such that   
    \begin{align}
        |f(\gamma_\ep)-f(y)|
        \leq C\ep \label{eqbd}
    \end{align}
    for small $\ep>0$.
    
    Moreover, since $\gamma$ is a future timelike geodesic, its "speed" $|\dot \gamma_t|_g$ is a positive constant.
    Thus, $\ell(y,\gamma_\ep)\geq \ep |\dot \gamma|_g$. Hence
    \begin{align}
       -c_s(y,\gamma_\ep)
        \geq s^{1-p} (\ep |\dot \gamma|_g)^p. \label{eqbe}
    \end{align}
    Since $p<1$, we can combine inequalities \eqref{eqbd} and \eqref{eqbe} to obtain 
    \[
        \hat T^{V_0}_sf(y) \geq f(\gamma_\ep)-c_s(y,\gamma_\ep)> f(y)
    \]
    for sufficiently small $\ep$.
    
        \item   Let $z\in V_0$ be optimal. Suppose, for contradiction, that $\ell(y,z)=0$, and let $\gamma:[0,s]\to M$ be a future causal null geodesic connecting $y$ with $z$. By Part (i), $y\neq z$, so $\gamma$ is non-constant. Hence, we can choose $\xi_s\in T_{z}M$ with $g_z(\dot \gamma_s,\xi_s)<0$. Let $\xi:[0,s]\to TM$ be the parallel transport of $\xi_s$ along $\gamma$, and consider a smooth variation $\sigma:(-\ep,\ep)\times [0,s]\to M$ of $\gamma$ with variational vector field $\tilde \xi_t=t\xi_t$, fixing $\sigma(r,0)=\gamma_0=y$ for all $r\in (-\ep,\ep)$. Our goal is to show that, for $r>0$ small enough,
    \begin{align}
        f(\sigma(r,s))-c_s(y,\sigma(r,s))>f(z)-c_s(y,z), \label{eqx}
    \end{align}
    contradicting the optimality of $z$, thus proving the claim. 

    By the compatibility of the connection with the metric, we compute for all $t\in [0,s]$
    \begin{align*}
        \frac{d}{dr}\Big\vert_{r=0} g(\partial_t \sigma(r,t),\partial_t \sigma(r,t))
        &= 2 g\bigg(\frac{D}{dr}\Big\vert_{r=0}\frac{\partial \sigma}{\partial t}(r,t),\frac{\partial \sigma}{\partial t}(0,t)\bigg)
        \\[10pt]
        &=
        2 g\bigg(\frac{D}{dt}\frac{\partial \sigma}{\partial r}(0,t),\dot \gamma_t\bigg)
        \\[10pt]
        &=
        2 g\bigg(\frac{D}{dt}(t\xi_t),\dot \gamma_t\bigg)
        \\[10pt]
        &= 2 g(\xi_t,\dot \gamma_t)
        \\[10pt]
        &= 2 g(\xi_s,\dot \gamma_s)=:-2a<0,
    \end{align*}
    where in the step to the last line we used the fact that $\xi$ and $\dot \gamma$ are parallel along $\gamma$.
    Thus, Taylor expansion and the fact that $\gamma$ is a null geodesic yield
    \begin{align}
        g(\partial_t \sigma(r,t),\partial_t \sigma(r,t))\leq -|\dot \gamma_t|^2_g -2ar +O(r^2)\leq -ar \label{hhhh}
    \end{align}
    for small values $r>0$ and all $t\in [0,s]$. In particular, $\sigma(r,\cdot)$ is either future timelike or past timelike for small $r$.
    Since $y\neq z$ and $z\in J^+(y)$ we must have $z\notin J^-(y)$. Thus $z\approx \sigma(r,s)\notin J^-(y)$ for small $r$, so that $\sigma(r,\cdot)$ is in fact future timelike for small $r$.
    This, together with
    \begin{align*}
        \ell_g(\sigma(r,\cdot))\geq s \sqrt{ar},
    \end{align*}
    as follows from \eqref{hhhh}, gives
    \begin{align*}
        \ell(y,\sigma(r,s))\geq  s\sqrt{ar}
    \end{align*}
    and therefore
    \begin{align}
        -c_s(y,\sigma(r,t))+c_s(y,z)= s^{1-p} \ell(y,\sigma(r,s))^p
        \geq s (ar)^\frac p2 \label{hwdauidjaias}
    \end{align}
    for small $r$.
    On the other hand, since $\partial_r \sigma(0,s)=s\xi_s$, it holds $d_h(\sigma(r,s),z)\leq 2s|\xi_0|_hr$, if $r$ is small. Thus, denoting by $L$ a local Lipschitz constant of $f$ near $z$, it follows for small $r$ that
    \begin{align*}
        |f(\sigma(r,s))-f(z)| \leq 2Ls\, |\xi_0|_hr.
    \end{align*}
    Combining with \eqref{hwdauidjaias}, we see that that \eqref{eqx} holds for small $r$.
    \item 
    Let $z\in V_0$ be any optimal point, so that $z\in I^+(y)$ by Part (ii). Let $\gamma:[0,s]\to M$ be any maximizing geodesic connecting $y$ to $z$. Theorem \ref{A} states that $c_s$ is super-differentiable at $(y,z)$ with
    \[
        \Big(-\frac{\partial L}{\partial v}(y,\dot \gamma_0),\frac{\partial L}{\partial v}(z,\dot \gamma_s)\Big) \in \partial^+ c_s(y,z).
    \]
    It is then a simple consequence of the definition in \eqref{eqzaa} and the optimality of $z$ that
    \begin{align*}
        \frac{\partial L}{\partial v}(y,\dot \gamma_0)\in \partial^- \hat T_s^{V_0}f(y) \quad \text{ and } \quad \frac{\partial L}{\partial v}(z,\dot \gamma_s)\in \partial^+f(z).
    \end{align*}
     In particular, since $\hat T_s^{V_0}f$ is differentiable at $y$, its derivative must coincide with the unique sub-differential, and hence
    \[
        d_y(\hat T^{V_0}_sf) = \frac{\partial L}{\partial v}(y,\dot \gamma_0).
    \]
    It remains to prove the uniqueness of the maximizing curve $\gamma$ (and hence also of $z$). Since the Legendre transform is a diffeomorphism, the preceding identity shows that $\dot \gamma_0$, and hence $\gamma$, is uniquely determined by the differential of $\hat T^{V_0}_sf$ at $y$. This proves the uniqueness, since $z$ was any optimal point and $\gamma$ was any maximizing geodesic from $y$ to $z$. \qedhere
    \end{enumerate}
\end{proof}

\begin{lemma}[Uniform control of maximizers]\label{lk}
    Let $K_0,V_0$ and $\mathcal{F}$ as in Remark \ref{rmzd}, and let $V\Subset V_0$. Then there exist constants $C$ and $s_0>0$ such that, for all $f\in \mathcal{F}$, $s\in (0,s_0]$ and $y\in V$, the following inequalities hold
    \begin{align}
        \hat T^{V_0}_sf(y)
        \geq f(y)> \sup\{f(z)-c_s(y,z)\mid z\in V_0,\ d_{h}(y,z)\geq Cs^{1-p}\}. \label{eqzh}
    \end{align}
    In fact, $C$ and $s_0$ can be chosen such that the following stronger assertion holds: If $d_h(y,z)\geq Cs^{1-p}$ for some $z\in V_0$, then
    \[
        f(y) \geq f(z)-c_s(y,z) +\frac{C_0}{2}d_h(y,z),
    \]
    where $C_0=C_0(K_0,V_0)$.
\end{lemma}
\begin{proof}
    The first inequality in \eqref{eqzh} is obvious; the second inequality follows from the more general second part of the lemma, which we shall prove now. 
    
    Set $C:=\frac{2M^p}{C_0}$, where 
    \begin{align*}
        M:=\sup\{\ell(y,z)\mid y,z\in \overbar V_0, z\in J^+( y)\}
    \end{align*}
    is finite thanks to the compactness of $\overbar V_0$. Now choose $s_0>0$ with $Cs_0^{1-p}\leq d_h(V,\partial V_0)$. Fix $f,s$ and $y$ as in the lemma, and let $z\in V_0$ be such that $d_h(y,z)\geq Cs^{1-p}$. 
    
    Let $\gamma:[0,s]\to M$ be a maximizing geodesic connecting $y$ to $z$. Since $V_0$ is causally convex, it follows that $\gamma_t\in V_0$ for all $t\in [0,s]$. In particular,  $\frac{d}{dt} (f\circ \gamma)(t)\leq -C_0|\dot \gamma_t|_h$ for a.e.\ $t$. Thus
    \begin{align*}
        f(z)-f(y)=\int_0^s \frac{d}{dt} (f\circ \gamma)(t)\, dt
        \leq
        -C_0\, d_h(y,z).
    \end{align*}
    Thus, using $d_h(y,z)\geq \frac 12 d_h(y,z)+\frac{M^p}{C_0}s^{1-p}$, we obtain
    \[
        f(z)-c_s(y,z)
        \leq 
        f(y)-C_0\, d_h(y,z)+s^{1-p}M^p
        \leq f(y) -\,\frac{C_0}{2} d_h(y,z). \qedhere
    \]
\end{proof}

\begin{corollary}[Existence of optimal points] \label{c1}
    Let $K_0,V_0$ and $\mathcal{F}$ be as in Remark \ref{rmzd}, and let $V\Subset V_0$. Then there exists $s_0>0$ such that, for any $f\in \mathcal{F}$, $s\in (0,s_0]$ and $y\in V$, the supremum in the definition of $\hat T_s^{V_0}f(y)$ is attained at some $z\in V_0$, and necessarily $z\in I^+(y)$. 
\end{corollary}
\begin{proof}
     Fix $C,s_0>0$ as in Lemma \ref{lk}. Choosing $s_0$ even smaller, we may assume that $d_h(V,\partial V_0)\geq Cs_0^{1-p}$.
     
     Fix $(s,y)\in (0,s_0]\times V$, and let $z_k$ be a maximizing sequence in the definition of $\hat T^{V_0}_sf(y)$. Notice that obviously $z_k\in J^+(y)$. By Lemma \ref{lk} and the completeness of the metric $h$, it follows that, up to some (non-relabeled) subsequence, $z_k\to z\in J^+(y)\cap V_0$ (having used the closedness of $J^+$). The continuity of both $f$ on $V_0$ and $\C$ on $(0,\infty)\times J^+$ imply
    \begin{align*}
        \hat T_s^{V_0}f(y)
        =\lim_{k\to \infty} (f(z_k)-c_s(y,z_k))
        =
        f(z)-c_s(y,z).
    \end{align*}
    This shows that the supremum is attained at some $z$. The fact that $z$ must belong to $I^+(y)$ follows from Lemma \ref{lbb}(ii).
\end{proof}

The following theorem partly summarizes and generalizes the previous results; it states that that the optimal points $z$, as $y$ varies over $V$, remain uniformly bounded (Lemma \ref{lk}) and stay even uniformly bounded away from the null future of $y$. Since we already know that for each $f$, $s$ and $y$, optimal points must belong to the chronological future of $y$, the proof is carried out by a compactness argument.

\begin{theorem}[Uniform semiconvexity]\label{lzb}
    Let $K_0,V_0$ and $\mathcal{F}$ be as in Remark \ref{rmzd}, and let $V\Subset V'\Subset V_0$. Then there exist $C,s_0>0$ such that the following holds: Whenever $\mathcal{F}'\subseteq \mathcal{F}$ is a subset which is compact in the $C(\overbar V',\R)$-topology, then there exists a non-decreasing sequence of $(\delta_s)_{0<s\leq s_0}$ of positive numbers such that, for any $f\in \mathcal{F}',s\in (0,s_0]$ and $y\in V$, it holds
     \begin{align}
        \hat T^{V_0}_sf(y)>\sup\{f(z)-c_s(y,z)\mid z\in V_0,\ (d_h(y,z)\geq Cs^{1-p} \text{ or } \ell(y,z)\leq \delta_s)\}. \label{eqaee}
    \end{align}
    In particular, for each $f\in \mathcal{F}$ and $s\in (0,s_0]$, $\hat T_s^{V_0}f$ is locally semiconvex on $V$, and the family of functions
    \begin{align*}
        \{\hat T_s^{V_0}f\mid f\in \mathcal{F}',\ s\in [s_1,s_0]\}
    \end{align*}
    is even uniformly locally semiconvex on $V$ for any $s_1\in (0,s_0]$.
\end{theorem}

\begin{proof}
    Let $s_0>0$ and $C>0$ be such that Lemma \ref{lk} and Corollary \ref{c1} hold when applied to $V'\Subset V_0$ (instead of $V\Subset V_0$). We may assume that $Cs_0^{1-p}\leq d_h(V',\partial V_0)$. To prove \eqref{eqaee}, taking into account Lemma \ref{lk}, it suffices\footnote{
     Indeed, we can set $\delta_s:=\delta_{s_1}'$ for $s\in [s_1,s_0]$, $\delta_s:=\min\{\delta_{s_1/2}',\delta_{s_1}'\}$ for $s\in [s_1/2,s_1]$ and so on.
     } 
     to show that, if $s_1\in (0,s_0]$, there exists $\delta_{s_1}'>0$ such that, for all $f\in \mathcal{F'}$, $s\in [s_1,s_0]$ and $y\in V$, it holds
     \[
        \hat T_s^{V_0}f(y) > \sup\{f(z)-c_s(y,z)\mid  z\in V_0,\ d_h(y,z)\leq Cs^{1-p},\ \ell(y,z)\leq \delta_{s_1}'\}.
     \]
     Suppose this was not the case; then we find sequences $f_k\in \mathcal{F'}$, $y_k\in V$ and $s_k\in [s_1,s_0]$, and for every $k\in \N$ some $z_k\in V_0\cap J^+(y_k)$ with $\ell(y_k,z_k)\leq 1/k$ and $d_h(y_k,z_k)\leq Cs_0^{1-p}$ such that
     \[
        f_k(z_k)-c_{s_k}(y_,z_k)\geq \hat T_s^{V_0}f_k(y_k)-\frac 1k.
     \]
     Up to passing to subsequences, we may assume that $f_k\to f\in \mathcal{F'}$ in $C(\overbar V',\R)$, $y_k\to y\in  V'$, $s_k\to s\in [s_1,s_0]$ and $z_k\to z\in V_0\cap J^+(y)\subseteq V_0$ with $\ell(y,z)=0$ (having used closedness of $J^+$ and continuity of $\ell$).

     It follows from the convergence of $f_k$ to $f$ and the continuity of both $f$ on $V_0$ and $\C$ on $(0,\infty)\times J^+$ that
     \begin{align*}
        \hat T_s^{V_0}f(y) \geq f(z)-c_{s}(y,z) 
        =\lim_{k\to \infty} (f_k(z_k)-c_{s_k}(y_k,z_k))
        = 
        \lim_{k\to \infty} 
        \hat T_{s_k}^{V_0} f_k(y_k).
     \end{align*}
     We claim that this inequality must hold as an equality. Indeed, let $z'\in I^+(y)$ be optimal for $\hat T_s^{V_0}f(y)$, and notice that then also $z'\in I^+(y_k)$ for large $k$. Hence, by $f_k\to f$ and the continuity of $\C$ on $(0,\infty)\times J^+$,
    \begin{align*}
        \liminf_{k\to \infty} \hat T_{s_k}^{V_0}f_k(y_k)
        \geq \liminf_{k\to \infty}\,  [f_k(z')-c_{s_k}(y_k,z')] &= f(z')-c_{s}(y,z')
        \\
        &=\hat T_s^{V_0}f(y).
    \end{align*}
     This proves the claim, and the claim implies that $z$ is optimal. This, together with $\ell(y,z)=0$, however, stands in contradiction to Lemma \ref{lbb}(ii). We thus have proved \eqref{eqaee}. Let us now show how the statement about the semiconvexity follows.

     Fix $s_1\in (0,s_0]$. First note that the existence of an optimal $z$ for each $y\in V$ (Lemma \ref{lzb}) ensures finiteness of $\hat T_s^{V_0}f$ on $V$. Let $y\in  V$ be arbitrary; we show the uniform local semiconvexity near $y$.  
     
     Note that \eqref{eqaee}, the continuity of $\ell$ and the completeness of $h$ easily guarantee an open neighbourhood $\tilde V\subseteq V$ of $y$ and a compact set $K\subseteq M$ with $\tilde V\times K\subseteq I^+$ such that, for all $\tilde y\in \tilde V$
    \[
        \hat T_s^{V_0}f(\tilde y) = \sup\{f(z)-c_s(\tilde y,z)\mid z\in V_0\cap K\}.
    \]
Since $\C$ is locally semiconcave on $(0,\infty)\times I^+$, it follows easily from the compactness of $K$ that the family $(f(z)-c_s(y',z))_{s\in [s_1,s_0],\ f\in \mathcal{F}',\ z\in (V_0\cap K)}$ is uniformly locally semiconvex on $\tilde V$ (\cite{Fathi/Figalli}, Appendix A); hence the family $\hat T_s^{V_0}f$, $s\in [s_1,s_0],\ f\in \mathcal{F}'$, is uniformly locally semiconvex on $\tilde V$ as finite suprema of uniformly locally semiconvex functions (again \cite{Fathi/Figalli}, Appendix A).
\end{proof}

\subsection{Semiconcavity}

In this section the main goal is to prove the counterpart to the second part of Theorem \ref{lzb}, that is, we show local semiconcavity of the local Lax-Oleinik semigroup: 

\begin{theorem}\label{c2}
Let $K_0,F_0$ and $\mathcal{F}$ be as in Remark \ref{rmzd}, and let $V\Subset V_0$. Suppose that the family $\mathcal{F}'\subseteq \mathcal{F}$ is uniformly locally semiconcave on $V_0$. Then there exists $s_0>0$ such that the family 
\[
    \{\hat T_s^{V_0}f\mid s\in [0,s_0],\ f\in \mathcal{F'}\}
\]
is uniformly locally semiconcave on $V$.
\end{theorem}

Although slightly modified and generalized to the case of a family of functions, the Riemannian version of this theorem is the key step in Bernard’s Lasry–Lions-type result \cite{Bernard}. A detailed proof in the non-compact Riemannian setting – from which we took the idea of first studying some local semigroup – is given in \cite{Fathi/Figalli/Rifford}, which also considers a whole family of functions and whose approach our method builds on.

\begin{lemma}[Approximation by smooth functions]\label{pro1}
    Let $K_0,V_0$ and $\mathcal{F}$ be as in Remark \ref{rmzd}. Let $\mathcal{F}'\subseteq \mathcal{F}$ be a family of functions that are uniformly locally semiconcave on $V_0$, and fix $y_0\in V_0$. Then there exists a causally convex and open precompact neighbourhood $V_0'\subseteq V_0$ of $y_0$, for each $f\in \mathcal{F}'$ an index set $I_f$, and a family of smooth functions $g_{i,f}:V_0'\to \R$, $i\in I_{f}$, $f\in \mathcal{F}'$, such that 
    \begin{enumerate}[(i)]
        \item 
        There exists a compact set $K_0'\subseteq \op{int}(\C^*)$ such that
        \[
        \{(y,d_yg_{i,f})\mid y\in  V_0',\ f\in \mathcal{F}',i\in I_{f}\}\subseteq K_0'.
        \]
        In particular, $g_{i,f}\in \mathcal{F}(K_0',V_0')$.
        \item 
        If $h^*$ denotes any Riemannian metric on $T^*M$, then the functions $dg_{i,f}:V_0'\to T^*M,\ z\mapsto (z,d_zg_{i,f})$, are equi-Lipschitz w.r.t.\ the metrics $d_h$ and $d_{h^*}$.
        
        \item 
         $\forall f\in \mathcal{F}'$: $f_{|V_0'}= \inf_{i\in I_{f}} g_{i,f}$.

        \item 
         $\forall f \in \mathcal{F}',\ y\in V_0',\ p\in \partial^+f(y)$: $\exists\,  i\in I_f$: $g_{i,f}(y)=f(y)$ and $d_yg_{i,f}=p$.
    \end{enumerate}
\end{lemma}

\begin{proof}
    It is well-known that the set of precompact open and causally convex sets form a basis of the topology. Hence, by assumption, there exists a chart $(\phi,V_0')$ around $y_0$ with $V_0'\subseteq V_0$ being precompact and causally convex such that, for some $C>0$,
     \begin{align}
        \forall f \in \mathcal{F}':\ f\circ \phi^{-1} \text{ is $C$-concave and $C$-Lipschitz}. \label{eqn}
     \end{align}
    For each $f\in \mathcal{F}'$, $y\in V_0'$ and $p\in \partial^+f(y)$, we define a smooth function $ g_{y,p,f}:V_0'\to \R$ by
     \begin{align*}
         &g_{y,p,f}\circ \phi^{-1}(x):= f(y)+p\circ d_{\phi(y)}\phi^{-1}(x-\phi(y))-C|x-\phi(y)|^2.
     \end{align*} 
     By \eqref{eqzd}, we may choose $V_0'$ smaller (and restrict $\phi$) such that there exists a compact set $K_0\subseteq K_0'\Subset \op{int}(\C_x)$ with
     \begin{align}
       (y',d_{y'}g_{y,p,f})\in K_0' \qquad \forall y'\in V_0',\ f\in \mathcal{F}',\ p\in \partial^+ f(y). \label{eqze}
    \end{align}
    We then define, for each $f\in \mathcal{F}'$, the index set
		\begin{align*}
		I_f:=\{(y,p)\mid y\in V_0',\ p\in \partial^+f(y)\},
		\end{align*}
		and we consider the family of smooth functions $(g_{i,f})_{i\in I_f, f\in \mathcal{F}'}$. 

        Part (i) follows from \eqref{eqze}. For Part (ii) notice that, upon a restricting $V_0'$ one more time, we may assume that $\phi:V_0'\to \phi(V_0')\subseteq \R^n$ is Lipschitz and that the tangent bundle chart $T^*\phi:T^*V_0'\to \R^n\times \R^n$ restricts to a bi-Lipschitz map on the precompact set $T^*V_0'\cap K_0'$. Therefore it suffices to show that the functions 
        \[
            \phi(V_0')\to \R^n,\ x\mapsto D(g_{j,f}\circ \phi^{-1})(x),
        \]
        are equi-Lipschitz. But this is trivial by definition, the Lipschitz constant being $2C$. For part (iii), note that the $C$-concavity of $f\circ \phi^{-1}$ implies $g_{i,f}\geq f_{|V_0'}$, and thus "$\leq$" in (iii). The reverse relation "$\geq$" follows from (iv) (which is evident by definition of $g_{i,f}$) since the super-differential of a locally semiconcave function at any point is not empty.
\end{proof}

\begin{remark}\rm
    Since $g_{i,f}\in \mathcal{F}(K_0',V_0')$, the natural local Lax-Oleinik semigroup to study is $\hat T_s^{V_0'}g_{i,f}$.
\end{remark}

\begin{lemma}
    Let $F:(-a,a)\times U\to M$ be a smooth function defined on an open subset $(-a,a)\times U$ of $\R\times T^*M$, and suppose that $F(0,z,p)=z$ for all $(z,p)\in U$. If $K\subseteq U$ is any compact set and $h^*$ is any Riemannian metric on $T^*M$, then there exists a modulus of continuity $\omega_F(s)$ such that
    \[
        d_h(F(s,z,p),F(s,z',p')) \geq d_h(z,z') -\omega_F(s)\, d_{h^*}((z,p),(z',p'))
    \]
    for all  $(s,z,p)\in (-a,a)\times K$.
\end{lemma}
\begin{proof}
    The proof is easy. First one shows that, given $\ep>0$, it suffices to prove the lemma for all $|s|\leq \ep$. Then the problem is local, so one may assume $M=\R^n$. Finally, consider $F(s,z,p)-z$ and use Taylor's theorem.
\end{proof}

For the next theorem, recall that $\psi_s$ denotes the Hamiltonian flow, $\phi_s$ the Lagrangian (i.e.\ geodesic) flow and that $\mathcal{L}$ denotes the Legendre transform.

\begin{theorem}[Propagation of semiconcavity]\label{thmd}
    Under the hypothesis and notation of Lemma \ref{pro1}, given two open sets $V'\Subset V_0''\Subset V_0'$, there exists $s_0>0$ such that, for any $s\in [0,s_0]$, $f\in \mathcal{F}'$ and $i\in I_{f}$, the map 
    \[
    \psi_{i,f,s}:V_0''\to M,\ \psi_{i,f,s}(z):=\pi^*\circ \psi_{-s}(z,d_zg_{i,f}),
    \]
    where $\pi^*:T^*M\to M$ is the canonical projection, is well-defined and a homeomorphism onto its image which contains $V'$. 
    Moreover, the family of maps 
 \[
 \{\hat T_s^{V_0'}g_{i,f}\mid s\in [0,s_0], f\in \mathcal{F}', i\in I_{f}\}
 \]
 is uniformly locally semiconcave on $V'$. Finally, we have
\begin{align}
    \hat T_s^{V_0'}g_{i,f}(\psi_{i,f,s}(z))=g_{i,f}(z)-
    c_s(\psi_{i,f,s}(z),z), \label{eqda}
\end{align}
 whenever $\psi_{i,f,s}(z)\in V'$. 
\end{theorem}
\begin{proof}  
    Define the index set $J:=\{(i,f)\mid f\in \mathcal{F}',\ i\in I_{f}\}$, so that $g_{i,f}=g_j$ and $\psi_{i,f,s}=\psi_{j,s}$.
    
    We start with the following observation:
    If $s>0$ and $\hat T^{V_0'}_sg_j$ is differentiable at $y\in V'$, then Lemma \ref{lbb} states that there exists at most one $z\in V_0'$ with
    \begin{align}
        \hat T^{V_0'}_sg_j(y)=g_j(z)-c_s(y,z), \label{j}
    \end{align}
    and necessarily $z\in I^+(y)$. Moreover, Part (iii) of the same lemma states that
    \begin{align*}
        d_y(\hat T^{V_0'}_sg_j)=\frac{\partial L}{\partial v}(y,\dot \gamma_0)\quad \text{ and } \quad d_zg_j=\frac{\partial L}{\partial v}(z,\dot \gamma_s),
    \end{align*}
     where $\gamma:[0,s]\to M$ is the unique maximizing geodesic from $y$ to $z$. Thus, since the Hamiltonian flow is conjugate to the geodesic flow via the Legendre transform $\mathcal{L}$ (Remark \ref{reme}), it follows
    \begin{align}
        (y,d_y(\hat T^{V_0}_sg_j))=\Bigl(\gamma_0,\frac{\partial L}{\partial v}(\gamma_0,\dot \gamma_0)\Bigr) 
        &={\cal L}(\gamma_0,\dot \gamma_0) \nonumber
        \\
        &={\cal L}(\phi_{-s}(\gamma_s,\dot \gamma_s)) \nonumber
        \\
        &=\psi_{-s}\Bigl(z,\frac{\partial L}{\partial v}(z,\dot \gamma_s)\Bigr) \nonumber
        \\
        &=\psi_{-s}(z,d_zg_j). \label{eqy}
    \end{align}
    Having made this observation, let us start start with the actual proof.
    \\

    Fix a Riemannian distance $h^*$ on $T^*M$. By Lemma \ref{pro1}(i),(ii), the set
    \begin{align}
        K_0'':=\overbar{{\{(z,d_zg_j)\mid z\in V_0'',\  j\in J\}}}\subseteq T^*V_0', \label{eqzp}
    \end{align}
    is a compact subset of $\op{int}(\C^{*})\cap T^*V_0'$, and the mappings
    \begin{align}
        dg_j:V_0''\to T^*M,\ z\mapsto (z,d_zg_j), \label{eqzq}
    \end{align}
    admit a uniform Lipschitz constant $C_1$.
    From \eqref{eqzp}, we conclude the existence of a precompact open neighbourhood $U\Subset \op{int}(\C^{*})\cap T^*V_0'$ of $K_0''$, and a time $T>0$  such that the Hamiltonian flow $\psi$ is well-defined on $[-T,T]\times U$. Let us consider the map
    \begin{align*}
         F:(-T,T)\times U\to M,\ (s,z,p)\mapsto (\pi^*\circ \psi_{-s})(z,p).
    \end{align*}
     $F$ is smooth and  satisfies $F(0,z,p)=z$. Consequently, by the above lemma, there exists a modulus of continuity $\omega_F:[0,\infty)\to [0,\infty)$ such that, for any $s\in (-T,T)$ and $(z,p)\in K_0''$, we have
     \begin{align}
        d_h(F(s,z,p),F(s,z',p')) \geq d_h(z,z')-\omega_F(s)\, d_{h^*}((z,p),(z',p')). \label{eqzt}
    \end{align}
     Let $C,s_0>0$ be as in Lemma \ref{lk}, Corollary \ref{c1} and Theorem \ref{lzb} when applied to $\mathcal{F}(K_0',V_0')$ and $V'\Subset V_0'$. We may assume that
     \begin{align*}
        s_0< T,\quad 1-C_1\omega_F(s_0)>0 \quad \text{ and } \quad Cs_0^{1-p}\leq d_{h}(V',\partial V_0'')/2.
     \end{align*}
    Now, fix $s\in (0,s_0]$ and $j\in J$.

    We come to the proof of the first part: Notice that $\psi_{j,s}(z)=F(s,z,d_zg_j)$ is well-defined on $V_0''$ by definition of $K_0''$, $U$ and $s_0<T$. Moreover, the definition of $C_1$ and \eqref{eqzt} imply that 
    \begin{align}
        d_h(\psi_{j,s}(z),\psi_{j,s}(z')) \geq (1-\omega_F(s) C_1) d_h(z,z') \quad \forall z,z'\in V_0'', \label{eqzu}
    \end{align}
    implying that $\psi_{j,s}$ is a homeomorphism onto its image – again by definition of $s_0$.
    To conclude the first part of the theorem, we shall now prove that $\psi_{j,s}(V_0'')$ contains $V'$:
    
    Let $y\in V'$ be a differentiability point of $\hat T_s^{V_0'}g_j$. By Corollary \ref{c1}, Lemma \ref{lk}, and the choice of $s_0$, there exists an optimal $z\in V_0'$ and $z$ belongs to the compact set
    \[
        A:=\{z'\in V_0'\mid d_h(z',\partial V_0'')\geq d_h(V',\partial V_0'')/2\}\subseteq V_0''.
    \]
    Moreover, \eqref{eqy} implies that $\psi_{j,s}(z)=y$. In particular, $\psi_{j,s}(A)$ contains a dense set of $V'$, namely all the differentiability points of the locally semiconvex (hence locally Lipschitz) function $\hat T^{V_0'}_sg_j$ (Theorem \ref{lzb}). However, since $\psi_{j,s}$ is continuous and $A$ is compact, it follows that actually $\psi_{j,s}(A)\supseteq V'$, proving the first part of the theorem.
    
   For the second part, notice that \eqref{eqy} shows for any differentiability point $y\in V'$ that
    \begin{align}
        (y,d{\hat T^{V_0'}_sg_j}(y))=\psi_{-s}\circ dg_j\circ \psi_{j,s}^{-1}(y) \label{eqzab}
    \end{align}
    The right-hand side of the above equation is well-defined and continuous on $V'$. Thus, being $\hat T^{V_0'}_sg_j$ locally semiconvex with derivative admitting a continuous extension to $V'$, it must be $C^1$ on $V'$. Thus, a posteriori, \eqref{eqzab} holds in fact at every point of $V'$. Not observe that right-hand side of \eqref{eqzab} is even Lipschitz continuous on $V'$ with Lipschitz constant 
    \begin{align*}
        \op{Lip}(\psi_{|[-T,T]\times K_0''}) C_1 (1-C_1\omega_F(s_0))^{-1}
    \end{align*}
    (having used \eqref{eqzq} and \eqref{eqzu}), which is independent of $s$ and $j$. To summarize, we showed that the functions 
    \[
        V'\to T^*M,\ y\mapsto (y,d_y{\hat T^{V_0'}_sg_j}),\quad s\in (0,s_0],\ j\in J,
    \]
    are equi-Lipschitz w.r.t.\ the Riemannian metrics $h$ and $h_*$ and have image in the compact set $\psi([-T,T]\times K_0'')\subseteq T^*M$ (this holds obviously also for $s=0$). It is not difficult to deduce from this that the family $\{\hat T^{V_0'}_sg_{j}\mid s\in [0,s_0],\ j\in J\}$ is uniformly locally semiconcave (and also semiconvex) on $V'$ (Lemma \ref{lema}), proving the second part of the lemma.

    Finally, to prove \eqref{eqda}, fix $s\in (0,s_0]$, $j\in J$ and $z\in V_0''$ such that $\psi_{j,s}(z)\in V'$. Then $\psi_{j,s}(z)\in V'$ is a differentiability point of $\hat T^{V_0'}_sg_j$, and hence our observation from the beginning of the proof shows that $\psi_{j,s}(z)=\psi_{j,s}(z')$ if $z'\in V_0''$ is the unique optimal point for $\hat T^{V_0'}_sg_j(\psi_{j,s}(z))$. This combines with the injectivity of $\psi_{j,s}$ to show $z'=z$ and hence
    \begin{align*}
        \hat T^{V_0'}_sg_j(\psi_{j,s}(z)) = g_j(z)-c_s(\psi_{j,s}(z),z).
    \end{align*}
    This concludes the proof.
\end{proof}

\begin{proof}[Proof of Theorem \ref{c2}]
    Fix $y_0\in V$, let $V_0'$ be any neighbourhood of $y_0$ as in Lemma \ref{pro1}, and let $y\in V'\Subset V_0''\Subset V_0'$ be any two open sets. By definition, it clearly suffices to show uniform local semiconcavity on $V'$. 
    
    Let $s_0>0$ be given by the theorem above and such that that Lemma \ref{lk} and Corollary \ref{c1} hold (with $s_0$ and some $C$) when applied to  $\mathcal{F}(K_0,V_0)$ and $V'\Subset V_0$. We may also assume that $Cs_0^{1-p}\leq d_h(V',V_0'')$. By the preceding theorem, it suffices to show that, for all $f\in \mathcal{F}'$ and $s\in [0,s_0]$,
    \[
        \hat T_s^{V_0}f = \inf_{i\in I_f} \hat T_s^{V_0'}g_{i,f} \quad \text{ on } V',
    \]
    since finite\footnote{Finiteness from $\hat T_s^{V_0}f$ on $V'$ follows from the existence of optimal points in Corollary \ref{c1}.} infima of uniformly locally semiconcave functions are again uniformly locally semiconcave (\cite{Fathi/Figalli}, Appendix A).
    
    Fix $f\in \mathcal{F}'$ and $s\in [0,s_0]$. If $s=0$, the result follows immediately from Proposition \ref{pro1}(iii). Thus, suppose $s\neq 0$, and let $y\in V'$. By Lemma \ref{lk}, Corollary \ref{c1} and the definition of $s_0$, there exists $z\in V_0''$ with 
    \begin{align*}
        \hat T_s^{V_0}f(y)=f(z)-c_s(y,z),
    \end{align*}
    and necessarily $z\in I^+(y)$.
    If $\gamma:[0,s]\to M$ is a necessarily future timelike maximizing geodesic with $\gamma_0=y$ and $\gamma_s=z$, Lemma \ref{lbb} implies that
    \begin{align}
        p:=\frac{\partial L}{\partial v}(z,\dot \gamma_s)\in \partial^+f(z). \label{eqag}
    \end{align}
    Therefore, Proposition \ref{pro1}(iv) ensures the existence of $i\in I_{f}$ with $g_{i,f}(z)=f(z)$ and $d_zg_{i,f}=p$. Since $\gamma$ is a geodesic with ${\cal L}(\gamma_s,\dot \gamma_s)=(z,p)$, we have $\gamma_t=\pi^*\circ \psi_{-s+t}(z,p)=\psi_{i,f,s-t}(z)$ for $t\in [0,s]$, and in particular, $\psi_{i,f,s}(z)=\gamma_0=y\in V'$.
       Hence, by the previous theorem, we obtain
    \begin{align*}
        \hat T_s^{V_0'}g_{i,f}(y)
        =
        g_{i,f}(z)-c_s(y,z)
        =
        f(z)-c_s(y,z)
        =
        \hat T_s^{V_0'} f(y).
    \end{align*}
    This shows that
    \begin{align*}
       \inf_{i\in I_{t}} \hat T^{V_0'}_sg_{i,f}(y)
       \leq \hat T^{V_0}_sf(y). 
    \end{align*}
    The converse inequality follows from $z\in V_0''\subseteq V_0'$ and Lemma \ref{pro1}(iii) since
    \begin{align*}
        \hat T_s^{V_0}f(y)=f(z)-c_s(y,z) \leq g_{i,f}(z)-c_s(y,z) \leq \hat T^{V_0'}_sg_{i,f}(y) \quad \forall i\in I_f. \qedhere
    \end{align*}
\end{proof}

\section{The (global) Lax-Oleinik semigroup}

\subsection{Proof of Theorem \ref{main}}

This section is devoted to the proof of the main theorem \ref{main}, which is a direct consequence of the following more general theorem.

\begin{theorem}[Lasry-Lions regularization]\label{thmza}
    Let $N$ be a smooth manifold, and let $u:N\times M\to \R\cup\{\pm\infty\},\ (x,y)\mapsto u(x,y)=u_x(y)$, be such that the following family of Lax-Oleinik functions
    \begin{align*}
        \hat u:(0,\infty)\times N\times M\to \R\cup\{\pm \infty\},\ (t,x,y)\mapsto (T_tu_x)(y), 
    \end{align*}
    is locally semiconcave on an open neighbourhood of $(t_0,x_0,y_0)$ with
    \begin{align}
        \partial^+(T_{t_0}u_{x_0})(y)\Subset \op{int}(\C_{y_0}^*). \label{eqzac}
    \end{align}
    Then there exist three open neighbourhoods $I$, $U$ and $V$ of $t_0$, $x_0$ and $y_0$, respectively, and some number $s_0>0$ such that the following properties hold:
    \begin{enumerate}[(i)]
    \item The family $\{\hat T_s\circ T_tu_x\mid (s,t,x)\in [0,s_0]\times I\times U\}$ is uniformly locally semiconcave on $V$, and for any $s_1\in (0,s_0]$, the family $\{\hat T_s\circ T_tu_x\mid (s,t,x)\in [s_1,s_0]\times I\times U\}$ is uniformly locally semiconvex on $V$. In particular, $\hat T_s\circ T_tu_x$ is $C^{1,1}_{loc}$ on $V$.
    \item For $(s,t,x,y)\in (0,s_0]\times I\times U\times V$, there exists a unique $z\in M$ with $\hat T_s\circ T_tu_x(y)=T_tu_x(z)-c_s(y,z)$, and necessarily $\ell(y,z)>0$. Moreover, setting $z=y$ if $s=0$, $z$ depends continuously on $(s,t,x,y)$.
    \end{enumerate}
\end{theorem}
\begin{proof}
    If $t_k\to t_0$, $x_k\to x_0$, $y_k\to y_0$ and $p_k\in \partial^+(T_{t_k}u_{x_k})(y_k)$, then local semiconcavity (hence Lipschitzity) of $\hat u$ near $(t_0,x_0,y_0)$ implies boundedness of $(p_k)$ (cf.\ \cite{Fathi}, Proposition 7.3.4). Hence, up to a subsequence, $(y_k,p_k)\to (y_0,p_0)$ and the semiconvexity of $\hat u$ implies $p_0\in \partial^+(T_{t_0}u_{x_0})(y)$. This observation combines with Assumption \eqref{eqzac} to guarantee the existence of open neighbourhoods $I_0\subseteq (0,\infty)$ of $t_0$ and $U_0,V_0\subseteq M$ of $x_0$ and $y_0$ such that $\hat u$ is locally semiconcave and Lipschitz (with constant $L$) on $I_0\times U_0\times V_0$ with
        \begin{align*}
            \{(y,p)\in T^*M\mid  (t,x,y)\in I_0\times U_0\times V_0,\ p\in \partial^+(T_{t}u_x)(y)\}\subseteq K_0
        \end{align*}
        for some compact set $K_0\subseteq \op{int}(\C^{*})$.
    We may assume that $V_0$ is precompact and causally convex; thus, $T_tu_x\in \mathcal{F}:=\mathcal{F}(K_0,V_0)$ for all $(t,x)\in I_0\times U_0$. 
        
        Now fix open neighbourhoods $V\Subset V'\Subset V_0$, $U\Subset U_0$ and $I\Subset I_0$ of $y_0$, $x_0$ and $t_0$, respectively. Notice that the family 
        \[
            \mathcal{F}':=(T_tu_x)_{(t,x)\in \overbar I\times \overbar U}\subseteq \mathcal{F}
        \]
        is compact in the $C(\overbar V',\R)$-topology, and also uniformly locally semiconcave on $V_0$ (\cite{Fathi}, Proposition A.17 and a compactness argument). We may now choose $C,s_0>0$ as in Lemma \ref{lk}, Corollary \ref{c1} and Theorems \ref{lzb}, \ref{c2}, and we may assume that \begin{align}
              Cs_0^{1-p}\leq d_h(V,\partial V_0),\quad t-s\in I_0 \quad \text{ and }\quad  Ls\leq \frac{C_0C}{4} s^{1-p} \label{eqdb}
        \end{align}
        for all $t\in I$ and $s\in [0,s_0]$, where $C_0:=C_0(K_0,V_0)$.
        \medskip 
        
        \noindent \textbf{Claim:} For $(s,t,x,y)\in (0,s_0]\times I\times U\times V$ it holds 
        \[
            \hat T_s\circ T_{t}u_x(y)>\sup\{T_{t}u_x(z)-c_s(y,z)\mid z\notin V_0\}.
        \] 
        In particular, $\hat T_s\circ T_tu_x(y)=\hat T_s^{V_0}\circ T_tu_x(y)$.
        \medskip 
        
        \noindent\textbf{Proof of claim:} Fix $(s,t,x,y)$ as in the claim, and let $z\in J^+(y)\backslash V_0$. Let $\gamma:[0,s]\to M$ be a maximizing geodesic connecting $y$ to $z$. By definition of $z$ and $s_0$, there exists $s'\in (0,s]$ with $d_h(y,\gamma_{s'})=Cs^{1-p}$; set $z':=\gamma_{s'}\in V_0$. It follows from Lemma \ref{semigroup} that
        \begin{align}
            T_tu_x(z)-c_s(y,z) &\leq T_{t-s+s'}u_x(z') +c_{s-s'}(z',z)-c_{s}(y,z) \nonumber
            \\[5pt]
            &=T_{t-s+s'}u_x(z')-c_{s'}(y,z'). \label{eqzc}
        \end{align}
        However, $t-s+s'\in I_0$ by \eqref{eqdb}, and thus $T_{t-s+s'}u_x\in \mathcal{F}$. Since in addition $d_h(y,z')=Cs^{1-p}$, we have 
        \begin{align*}
            T_tu_x(z)-c_s(y,z)
            &\overset{\eqref{eqzc}}{\leq} 
            T_{t-s+s'}u_x(z')-c_{s'}(y,z').
            \\[5pt]
            \text{by Lemma \ref{lk}}\qquad \qquad &\leq T_{t-s+s'}u_x(y)-\frac{C_0C}{2}s^{1-p}
            \\[5pt]
            \text{by definition of $L$} \qquad \qquad &\leq T_tu_x(y)+Ls-\frac{C_0C}{2} s^{1-p}
            \\[5pt]
            \eqref{eqdb} \qquad \qquad \qquad &\leq T_tu_x(y)-\frac{C_0C}{4}s^{1-p}
            \\[5pt]
            &\leq \hat T_s\circ T_tu_x(y)-\frac{C_0C}{4}s^{1-p}.
        \end{align*}
        Since $z\in J^+(y)\backslash V_0$ was arbitrary, we have proved the claim.
        \hfill \checkmark 
        \medskip
        
        The first two statements in (i) follows from the claim and Theorems \ref{lzb} and \ref{c2}. The second part of (ii) follows from the first one by a well-known result stating that a function is $C^{1,1}_{loc}$ if and only if it is locally semiconvex and locally semiconcave (\cite{Fathi}, Theorem 6.1.5). 
        
        It remains to prove (ii). 
        Note that the claim implies that $z$ is optimal for $\hat T_s\circ T_tu_x(y)$ if and only if it is optimal for $\hat T_s^{V_0}\circ T_tu_x(y)$; hence Corollary \ref{c1} and Lemma \ref{lbb} show existence and uniqueness of an optimal $z$, as well as the necessity of $z\in I^+(y)$. We only need to show that $z$ depends continuously on $(s,t,x,y)$. 
        
        Suppose that $(s_k,t_k,x_k,y_k)\in (0,s_0]\times I\times U\times V$ converges to $(s,t,x,y)\in (0,s_0]\times I\times U\times V$. Let $z_k$ and $z$ be the corresponding optimal points.
        Since $z\in I^+(y)$, also $z\in I^+(y_k)$ for large $k$, and thus the continuity of both $\hat u$ on $I\times U\times V$ and $\C$ on $(0,\infty)\times J^+$ imply
        \begin{align}
            \liminf_{k\to \infty}
            \hat T_{s_k}\circ T_{t_k}u_{x_k}(y_k)
            \geq \liminf_{k\to \infty} T_{t_k}u_{x_k}(z)-c_{s_k}(y_k,z) &= T_tu_x(z)-c_s(y,z) \nonumber
            \\
            &=\hat T_s\circ T_tu_x(y). \label{eqdc}
        \end{align}
        On the other hand, since $d_h(y_k,z_k)\leq Cs_0^{1-p}\leq d_h(V,\partial V_0)$, it follows that any subsequence $z_{j}$ of $z_k$ has a further subsequence $z_m$ that converges to some $\tilde z\in V_0\cap J^+(y)$. Then the continuity of both $\hat u$ on $I\times U\times V$ and $\C$ on $(0,\infty)\times J^+$ gives
        \begin{align*}
            \hat T_s\circ T_tu_x(y) \geq T_tu_x(\tilde z)-c_s(y,\tilde z) 
            )
            &=\lim_{m\to \infty} T_{t_m}u_{x_{m}}(z_{m})-c_{s_{m}}(y,z_{m})
            \\
            &= \lim_{m\to \infty} \hat T_{s_{m}}\circ T_{t_{m}}u_{x_{m}}(y_{m}).
        \end{align*}   
        Combining with \eqref{eqdc}, we see that equality must hold throughout, and thus $\tilde z=z$ is optimal.
        Since we have shown that any subsequence $z_j$ of $z_k$ has a subsequence $z_m$ that converges to $z$, the whole sequence must converge. This concludes the proof of (ii) if $s>0$. Suppose $s=0$. If $s_k=0$, the statement is trivial. If $s_k>0$, then the claim and Lemma \ref{lk} show that $d_h(y_k,z_k)\leq Cs_k^{1-p}$, and thus $z_k\to y$.
\end{proof}

\subsection{On the validity of the assumptions}

For a single function $u:M\to \overbar \R$, we can give a condition when the requirements of Theorem \ref{thmza} are satisfied. Recall the definition of causal compactness.

\begin{theorem}\label{m}
    Let $u:M\to \overbar \R$ be a function with $u\equiv +\infty$ outside of a causally compact set $K$, and let $y_0\in I^+(x_0)$ and $t_0>0$. Suppose that $u(x_0)$ and $T_{t_0}u(y_0)$ are finite. Additionally, assume that 
    $T_{t_0}u(y_0)=u(x_0)+c_{t_0}(x_0,y_0)$. Let $\gamma:[0,t_0]\to M$ be a maximizing geodesic connecting $x_0$ with $y_0$ and let $t_1\in (0,t_0)$. Then the Lax-Oleinik semigroup evolution $Tu$ is locally semiconcave on a neighbourhood of $(t_1,\gamma_{t_1})$, and differentiable at $(t_1,\gamma_{t_1})$ with
    \begin{align}
        d_{\gamma_{t_1}}(T_{t_1}u) \in \op{int}(\C_{\gamma_{t_1}}^*). \label{eqzn}
    \end{align}
    In particular, the hypothesis of Theorem \ref{main} are satisfies at $(t_1,\gamma_{t_1})$.
\end{theorem}

\begin{lemma}\label{xx}
      Under the assumptions of Theorem \ref{m}, even without the assumption on the existence of $K$, we have that $Tu$ is differentiable at $(t_1,\gamma_{t_1})$ with
    \[
        d_{\gamma(t_1)}(T_{t_1}u) \in \op{int}(\C_{\gamma_{t_1}}^*).
    \]
\end{lemma}
\begin{proof}
    Set $v:=T_{t_0}u$. We know from Lemma \ref{semigroup}(i) and (ii) that
    \begin{align}
        \hat T_{t_0-t}v = \hat T_{t_0-t}\circ T_{t_0-t}\circ T_{t} u\leq T_{t}u \label{eqzk}
    \end{align}
    for all $0\leq t\leq t_0$ However, since $\gamma$ is a maximizing geodesic, we also have
    \begin{align*}
        \hat T_{t_0-t_1}v(\gamma_{t_1}) - T_{t_1}u(\gamma_{t_1})
        &\geq \bigl(v(y_0) - c_{t_0-t_1}(\gamma_{t_1},y_0)\bigr) -\bigl(u(x_0)+c_{t_1}(x_0,\gamma_{t_1})\bigr)
        \\[5pt]
        & = c_{t_0}(x_0,y_0)- c_{t_0-t_1}(\gamma_{t_1},y_0)
        -c_{t_1}(x_0,\gamma_{t_1})
        \\[5pt]
        &=0.
    \end{align*}
    Combining this with \eqref{eqzk}, we conclude that $\hat T_{t_0-t_1}v(\gamma_{t_1}) = T_{t_1}u(\gamma_{t_1})$ and that 
    \begin{align}
        \hat T_{t_0-t_1}v(\gamma_{t_1}) &= v(y_0)-c_{t_0-t_1}(\gamma_{t_1},y_0) \quad \text{ and } 
        \\[5pt]
        T_{t_1}u(\gamma_{t_1}) &= u(x_0)+c_{t_1}(x_0,\gamma_{t_1}). \label{eqzl}
    \end{align}
    Since $\C(\cdot,x_0,\cdot)$ is super-differentiable at $(t_1,\gamma_{t_1})$ with
    \begin{align}
        \frac{\partial L}{\partial v}(\gamma_{t_1},\dot \gamma_{t_1})\in  \partial^+\C(t_1,x_0,\cdot)(\gamma_{t_1})\cap \op{int}(\C_{\gamma_{t_1}}^*)\label{eqai}
    \end{align}
    it follows from \eqref{eqzl} that also $Tu$ is super-differentiable at $\gamma_{t_1}$, and that  \eqref{eqai} is a super-differential for $T_{t_1}u$ at $\gamma_{t_1}$. 

    Similarly, we obtain that $(t,y)\mapsto \hat T_{t_0-t}v(y)$ is sub-differentiable at $\gamma_{t_1}$. This, however, implies by \eqref{eqzk} and $\hat T_{t_0-t_1}v(\gamma_{t_1}) = T_{t_1}u(\gamma_{t_1})$ that $Tu$ is also sub-differentiable at $(t_1,\gamma_{t_1})$. In particular, being both super- and sub-differentiable, $Tu$ (and hence $T_{t_1}u$) is differentiable at $(t_1,\gamma_{t_1})$, and the differential of $T_{t_1}u$ at $\gamma_{t_1}$ agrees with the unique super-differential \eqref{eqai}.
\end{proof}

\begin{lemma}[Local boundedness (cf.\ \cite{FathiHJ}, Theorem 6.2)] \label{www}
    Let $u:M\to \overbar \R$ be any function, and let $x_0,y_0\in M$ and $t_0>0$. Suppose that $u(x_0)$ and $T_{t_0}u(y_0)$ are finite. Then the function $Tu$ is locally bounded on the set $(0,t_0)\times (J^+(x_0)\cap J^-(y_0)).$
\end{lemma}
\begin{proof}
    Let $K\subseteq (0,t_0)\times (J^+(x_0)\cap J^-(y_0))$ be a compact set. The definition of the Lax-Oleinik semigroup yields for $(t,y)\in K$ that
    \begin{align}
         T_tu(y) \leq u(x_0)+\sup_{(t',y')\in K} c_{t'}(x_0,y'). \label{qwert1}
    \end{align}
    On the other hand, the semigroup property yields for $(t,y)\in K$ that
    \begin{align}
        T_{t_0}u(y_0)\leq T_tu(y)+c_{t_0-t}(y,y_0)\leq T_tu(y)+\sup_{(t',y')\in K} c_{t_0-t'}( y',y_0). \label{qwert}
    \end{align}
    Combining both inequalities immediately gives the claim. Notice that the suprema in \eqref{qwert1} and \eqref{qwert} are finite thanks to the continuity of $\C$ on $(0,\infty)\times J^+$ and the compactness of $K$.
\end{proof}

\begin{definition}[Calibrated curve]
    Let $u:M\to \overbar \R$ be any function, and let $\gamma:[0,b]\to M$ be a curve. We say that $\gamma$ is \emph{$u$-calibrated} if
    \begin{align*}
        T_tu(\gamma_t)=T_su(\gamma_s)+c_{t-s}(\gamma_s,\gamma_t) \text{ for any } 0\leq s \leq t \leq b.
    \end{align*}
\end{definition}

\begin{lemma}[Condition for calibratedness (cf.\ \cite{FathiHJ}, Proposition 7.5)] \label{w}
    Let $u:M\to \overbar \R$ be any function, and let $y_0\in J^+(x_0)$ and $t_0>0$. Suppose that $u(x_0)$ and $T_{t_0}u(y_0)$ are finite. Additionally, assume that 
    $T_{t_0}u(y_0)=u(x_0)+c_{t_0}(x_0,y_0)$. Let $\gamma:[0,t_0]\to M$ be a maximizing geodesic connecting $x_0$ with $y_0$.
    Then $\gamma$ is $u$-calibrated.
    An analogous result holds for the backward semigroup.
\end{lemma}
\begin{proof}
    Let $0\leq s\leq t\leq t_0$. By the semigroup property, we have
    \begin{align}
        &T_su(\gamma_s)\leq u(x_0)+c_s(x_0,\gamma_s), \nonumber
        \\[10pt]
        & T_tu(\gamma_t)\leq T_su(\gamma_s)+c_{t-s}(\gamma_s,\gamma_t), \label{x}
        \\[10pt]
        & T_{t_0}u(y_0)\leq T_tu(\gamma_t)+c_{t_0-t}(\gamma_t,y_0). \nonumber
    \end{align}
    Adding these three inequalities and canceling terms (noting that all involved quantities are finite by the above lemma), we get
    \begin{align*}
        T_{t_0}u(y_0)&\leq u(x_0)+c_s(x_0,\gamma_s)+c_{t-s}(\gamma_s,\gamma_t)+c_{t_0-t}(\gamma_t,y_0)
        \\[10pt]
        &=u(x_0)+c_{t_0}(x_0,y_0),
    \end{align*}
    where the last inequality follows from the fact that $\gamma$ is a maximizing geodesic.
    By assumption, this inequality must actually hold as an equality, which implies that equality must also hold in \eqref{x} (and also in the other two inequalities). This completes the proof of the lemma.

    The proof for the backward semigroup is completely analogous.
\end{proof}

\begin{lemma}\label{lad}
    Under the hypothesis of Theorem \ref{m}, there exists $\ep>0$, an open neighbourhood $V$ of $\gamma_{t_1}$ and an open interval $I$ around $t_1$ such that, for all $(t,y)\in I\times V$, we have
\begin{align*}
    T_tu(y)=\inf\{u(x)+c_{t}(x,y)\mid x\in K,\ \ell(x,y)\geq \ep\}.
\end{align*}
\end{lemma}
\begin{proof}
We argue by contradiction and assume that we find a sequence $(t_k,y_k)\to (t_1,\gamma_{t_1})$ and, for each $k$, some $x_k\in K$ such that the following properties hold:
    \begin{align}
         \ell(x_k,y_k)\to 0 \quad \text{ and } \quad 
        u(x_k) +c_{t_k}(x_k,y_k)-\frac{1}{k}\leq T_{t_k}u(y_k). \label{ww'}
    \end{align}
		By Lemma \ref{www}, $T_{t_k}u(y_k)$ is finite for large values of $k$, which implies $x_k\in J^-(y_k)$.
        
 Since $T_{t_0}u(y_0)=u(x_0)+c_{t_0}(x_0,y_0)$ and $\gamma$ is a maximizing geodesic from $x_0$ to $y_0$, it follows that $\gamma$ is $Tu$-calibrated (Lemma \ref{w}), i.e.\ 
    \begin{align}
        T_{t_0}u(y_0)=T_{t_1}u(\gamma_{t_1})+c_{t_0-t_1}(\gamma_{t_1},y_0).\label{poiu1}
    \end{align}
    On the other hand, the semigroup property of $Tu$ yields
    \begin{align}
        T_{t_0}u(y_0)\leq u(x_k)+c_{t_0}(x_k,y_0) \quad \text{ for all } k\in \N. \label{poiu2}
    \end{align}
    From \eqref{ww'},\eqref{poiu1} and \eqref{poiu2} we obtain
    \begin{align*}
        u(x_k)+c_{t_k}(x_k,y_k)-\frac 1k
        \leq &\  T_{t_k}u(y_k)-T_{t_1}u(\gamma_{t_1})+T_{t_1}u(\gamma_{t_1})
        \\
        \leq &\ 
        T_{t_k}u(y_k)-T_{t_1}u(\gamma_{t_1})+ u(x_k)
        \\
        &\ +c_{t_0}(x_k,y_0)-c_{t_0-t_1}(\gamma_{t_1},y_0).
    \end{align*}
   Rearranging this inequality leads to
    \begin{align*}
        c_{t_k}(x_k,y_k)+c_{t_0-t_1}(\gamma_{t_1},y_0) \leq T_{t_k}u(y_k)-T_{t_1}u(\gamma_{t_1})+c_{t_0}(x_k,y_0).
    \end{align*}
    Since $K$ is causally compact, along a subsequence that we do not relabel, we have $x_k\to x\in K$.
    Since $J^+$ is closed, it follows $x\in J^-(\gamma_{t_1})\subseteq I^-(y_0)$, and by the continuity of the Lorentzian distance, $\ell(x,\gamma_{t_1})=0$. Taking limits in the above inequality and regarding the continuity of $Tu$ at $(t_1,\gamma_{t_1})$, we conclude
    \begin{align*}
        c_{t_1}(x,\gamma_{t_1})+c_{t_0-t_1}(\gamma_{t_1},y_0)\leq c_{t_0}(x,y_0).
    \end{align*}
    Since the reverse inequality always holds, this must be an equality. This and Lemma \ref{lt} imply that the curve obtained by concatenating a maximizing geodesic from $x$ to $\gamma_{t_1}$ in time $t_1$ with the curve $\gamma_{|[t_1,t_0]}$ is still a maximizing geodesic. However, the first curve is not timelike, since $\ell(x,\gamma_{t_1})=0$. The second curve, on the other hand, is timelike because $(\gamma_{t_1},y_0)\in I^+$. Since geodesics cannot change their causal type, this is a contradiction.
\end{proof}

\begin{proof}[Proof of Theorem \ref{m}]
    We already proved \eqref{eqzn}. We claim that $Tu$ is locally semiconcave on $I\times V$. Indeed, fix $y\in V$, and notice that the previous result combines with the continuity of the Lorentzian distance function and the causal compactness of $K$ to guarantee the existence of an open neighbourhoods $V'\subseteq V$ around $y$ and of a compact set $K'\subseteq K$ with $V'\times K'\subseteq I^+$ such that, for all $(t,y')\in I\times V'$, it holds
    \begin{align*}
    T_tu(y)=\inf\{u(x)+c_{t}(x,y)\mid x\in K'\}. 
    \end{align*}
    Since $\C$ is locally semiconcave on $(0,\infty)\times I^+$, it easily follows from the compactness of $K'$ that $Tu$ is locally semiconcave on $I\times V'$ (\cite{Fathi/Figalli}, Appendix A). Since $y$ was arbitrary, the conclusion follows.
\end{proof}

\section{Application to optimal transport – Proof of Theorem \ref{thmg}}

\begin{lemma}\label{ll}
    Under the hypothesis and notation of Theorem \ref{thmg}, let $\tau\geq 0$ such that $t-s\geq \tau$ and $\tau\leq 1-t$. Then the pair
        $(t-s)^{p-1}(\hat T_{\tau}T_{s+\tau}\varphi,\hat T_{\tau}T_{t+\tau}\varphi)$
    is a $c_1$-subsolution. Moreover, whenever $\gamma:[0,1]\to M$ is a maximizing geodesic with 
    \[
        \psi(\gamma_1) -\varphi(\gamma_0) = c_1(\gamma_0,\gamma_1),
    \]
    then also
    \begin{align}
         (t-s)^{p-1}\Bigl(\hat T_{\tau}\circ T_{t+\tau}\varphi (\gamma_t)-\hat T_{\tau}\circ T_{s+\tau}\varphi(\gamma_s)\Bigr)=c_1(\gamma_s,\gamma_t). \label{eqzm}
    \end{align}
\end{lemma}
\begin{proof}
    The proof is very similar to the one of Lemma \ref{xx}. The subsolution property follows from easy manipulations of the Lax-Oleinik semigroups using Lemma \ref{semigroup}. The fact that equality holds at all pairs $(\gamma_s,\gamma_t)$ can be proved directly.
\end{proof}

\begin{proof}[Proof of Theorem \ref{thmg}]
   By assumption, the following identity holds for $\Pi$-a.e.\ $\gamma$:
   \begin{align}
        \psi(\gamma_1)-\varphi(\gamma_0)=c(\gamma_0,\gamma_1) \label{eqzj}
   \end{align}
   By the inner regularity of $\Pi$, we may choose a $\sigma$-compact set $A\subseteq \op{CGeo}([0,1],M)$ consisting of future timelike geodesics with $\Pi(A)=1$ and such that \eqref{eqzj} holds for all $\gamma\in A$.

    Let $\gamma\in A$. We claim that there exists $\tau=\tau(\gamma)>0$ such that $\hat T_\tau T_{s+\tau}\varphi$ and $\hat T_\tau T_{t+\tau}\varphi$ are $C^{1,1}_{loc}$ on open neighbourhoods $\Omega_{\gamma,s}$ of $\gamma(s)$ and $\Omega_{\gamma,t}$ of $\gamma(t)$, respectively. Indeed, since $\varphi(\gamma_0)$ is finite due to \eqref{eqzj}, this result follows immediately from Theorems \ref{thmza} and \ref{m} if we can show that
    \[
        T_1\varphi(\gamma_1) = \varphi(\gamma_0) +c_1(\gamma_0,\gamma_1)
    \]
    (here the assumption that $\Pi$ is concentrated on timelike geodesics is crucial to apply both Theorems).

    Indeed, observe that the fact that $(\varphi,\psi)$ is a $c_1$-subsolution implies that $\psi \leq T_1\varphi$, which combines with the definition of $A$ to yield
    \[
        T_1\varphi(\gamma_1) \geq \psi(\gamma_1) \overset{(\gamma\in A)}{=} \varphi(\gamma_0)+c(\gamma_0,\gamma_1) \geq T_1\varphi(\gamma_1),
    \]
    having used the definition of the Lax-Oleinik semigroup in the last inequality. This proves the claim.

    We may assume that $t-s\geq \tau(\gamma)$ and that $\tau(\gamma)\leq 1-t$. Choose locally finite smooth partitions of unity $(\rho_\gamma)_\gamma$ and $(\eta_\gamma)_\gamma$ of 
    \[
        \Omega_s:=\bigcup_{\gamma\in A} \Omega_{\gamma,s} \quad \text{ and } \quad  \Omega_t:=\bigcup_{\gamma\in A} \Omega_{\gamma,t},
    \]
    w.r.t.\ the open coverings $(\Omega_{\gamma,s})_\gamma$ and $(\Omega_{\gamma,t})_\gamma$.
    Define
    \begin{align*}
        \varphi_s(x):=
        \begin{cases}
            (t-s)^{1-p}\sum\limits_{\gamma \in A} \rho_\gamma(x)\, (\hat T_{\tau_\gamma} \circ T_{s+{\tau_\gamma}}\varphi(x)),& \text{ if } x\in \Omega_s,
            \\[5pt]
            \varphi_s(x):=+\infty,&\text{ if } x\notin \Omega_s,
        \end{cases}        
    \end{align*}
    and
    \begin{align*}
        \psi_t(y):=
        \begin{cases}
            (t-s)^{1-p}\sum\limits_{\gamma \in A} \eta_\gamma(y)\, (\hat T_{\tau_\gamma} \circ T_{t+{\tau_\gamma}}\varphi(y)),& \text{ if } y\in \Omega_t,
            \\[5pt]
            \psi_t(y):=-\infty,&\text{ if } y\notin \Omega_t.
        \end{cases}
    \end{align*}
    Clearly, $\mu_s(\Omega_s)=\mu_t(\Omega_t)=1$, and using the Lemma \ref{ll} it is easy to check that $(\varphi_s,\psi_t)$ is $(c_1,\pi_{s,t})$-calibrated. Being locally finite sums of $C^{1,1}_{loc}$ functions, $\varphi_s$ is $C^{1,1}_{loc}$ on $\Omega_s$ and $\psi_t$ is $C^{1,1}_{loc}$ on $\Omega_t$. This proves the first part of the theorem.

    To conclude the proof, suppose that $(\varphi,\psi)$ is optimal in the dual formulation, i.e.\ 
    \begin{align*}
        C(\mu,\nu) = \int \psi(y)\, d\mu_1(y) - \int \varphi(x)\, d\mu_0(x),
    \end{align*}   
    and let us show that $(\varphi_s,\psi_t)$ is optimal as well. Since $(\varphi_s,\psi_t)$ is $(c_1,\pi_{s,t})$-calibrated, it suffices to show that both functions are integrable w.r.t.\ $\mu_s$ and $\mu_t$, respectively, and integrate the equation defining a calibrated pair. Let us prove the integrability of $\varphi_s$, the integrability of $\psi_t$ can be checked analogously.
    
    Clearly, $\varphi_s$ is $\mu_s$-measurable since it is $C^{1,1}_{loc}$ on an open set of full $\mu_s$-measure. Since each $\gamma\in A$ is $Tu$-calibrated by Lemma \ref{w}, we have
    \[
        \varphi_s(\gamma_s)=\varphi(\gamma_0)+c_{s}(\gamma_0,\gamma_s)=\varphi(\gamma_0) +s c_1(\gamma_0,\gamma_1)\quad \text{for $\Pi$-a.e. $\gamma$.}
    \]
   This implies
    \begin{align*}
        \int |\varphi_s(x)|\, d\mu_s(x) 
        =
        \int |\varphi_s(\gamma_s)|\, d\Pi(\gamma)
        &\leq
        \int |\varphi(\gamma(0))| +s\, |c_1(\gamma(0),\gamma(1))|\, d\Pi(\gamma)
        \\[10pt]
        &=
        \int |\varphi(x)|\, d\mu_0(x) - s\, C(\mu_0,\mu_1)<\infty.
    \end{align*}
    Hence, $\varphi_s\in L^1(\mu_s)$.
\end{proof}

\appendix
\renewcommand{\thesection}{\Alph{section}}
\setcounter{section}{0}

\section*{Appendix} 

\section{Some simple proofs}\label{proofs}

\begin{proof}[Proof of Lemma \ref{lt}]
	This follows immediately from H\"older's inequality with exponents $1/p$ and $q=1/(1-p)$. Indeed, for any future causal curve $\gamma:[a,b]\to M$, we have
	\begin{align*}
	\mathbb{L}(\gamma) = - \int_a^b |\dot \gamma_t|_g^p\, dt \geq -(b-a)^{1-p} \ell_g(\gamma)^{p}.
	\end{align*}
	with equality if and only if $|\dot \gamma|_g$ is constant. This easily implies that a maximizing geodesic  is $L$-minimizing. Conversely, let $y\in I^+(x)$ and suppose $\gamma$ is $L$-minimizing (hence future causal). Being $(M,g)$ globally hyperbolic, Theorem \ref{asdfghj} guarantees the existence of a maximizing geodesic $\tilde \gamma:[a,b]\to M$ connecting $x$ to $y$. Then
	\begin{align*}
	\ell_g(\gamma)^p \geq -(b-a)^{-(1-p)} \mathbb{L}(\gamma) \geq -(b-a)^{-(1-p)} \mathbb{L}(\tilde \gamma)= \ell_g(\tilde \gamma)^p \geq \ell_g(\gamma)^p.
	\end{align*}
	Thus, $\gamma$ is maximizing as well, and equality must hold in each of the above steps, implying that $|\dot \gamma|_g$ is constant. Since $y\in I^+(x)$, we must have $|\dot \gamma|_g=cons.\neq 0$. Combining with the fact that $\gamma$ is a pregeodesic by the theorem above, we conclude that $\gamma$ is in fact a geodesic.
\end{proof}

\begin{proof}[Proof of Lemma \ref{lba}]
	For $p\in \op{int}(\C_x^{*})$, observe that 
	\begin{align*}
	pv-L(x,v)\xrightarrow{|v|_h\to \infty } -\infty.
	\end{align*}
	Indeed, if $v\notin \C_x$, we have $pv-L(x,v)=-\infty$. Othwerwise, since $p\in \op{int}(\C_x^{*})$, we can define
	\begin{align*}
	&\alpha:=\sup\{pv\mid v\in \C_x,|v|_{h}=1\}<0, \text{ and }
	\\[10pt]
	&\beta:=\inf\{L(x,v)\mid v\in \C_x,|v|_{h}=1\}>-\infty.
	\end{align*}
	Therefore, on $\C_x\backslash \{0\}$,
	\begin{align*}
	pv-L(x,v)\leq \alpha |v|_{h}-\beta |v|_{h}^p \xrightarrow{|v|_{h}\to \infty}-\infty.
	\end{align*}
	Thus, since continuous functions defined on compact sets attain their supremum, there is $v\in \C_x$ with
	\begin{align}
	H(x,p)=pv-L(x,v). \label{eqo}
	\end{align}
	We claim that $v\notin \partial \C_x$. First $v\neq 0$: If $v=0$, then we have, for any nonzero $w\in \C_x$,
	\begin{align*}
	p(\lambda w)-L(x,\lambda w)
	=
	\lambda pw+(\lambda|w|_g)^p
	> 0=pv-L(x,v)
	\end{align*}
	for sufficiently small $\lambda>0$, contradicting \eqref{eqo}.
	Now suppose $v\in \partial \C_x\backslash \{0\}$. Let $(e_0,...,e_n)$ be a generalized orthonormal frame in $T_{x}M$ with $e_0$ future timelike. Then
	\begin{align*}
	v=\sum_{i=0}^n \lambda_i e_i, \text{ with } \lambda_0>0 \text{ and }\lambda_0^2-\sum_{i=1}^n \lambda_i^2=0.
	\end{align*}
	However, if we define, for small $\ep>0$, 
	\begin{align*}
	v(\ep):=(\lambda_0+\ep)e_0+\sum_{i=1}^n \lambda_i e_i,
	\end{align*}
	then $v(\ep)$ is future causal and
	\begin{align*}
	|v(\ep)|_g^p= \left((\lambda_0+\ep)^2-\sum_{i=1}^n \lambda_i^2\right)^\frac p2 = (2\lambda_0 \ep+\ep^2)^\frac p2
	\geq (2\lambda_0)^\frac p2 \ep^\frac p2.
	\end{align*}
	Since $\lambda_0>0$, we conclude that
	\begin{align*}
	pv(\ep)-L(x,v(\ep))= pv(\ep)+|v(\ep)|_g^p
	\geq pv +\ep pe_0+(2\lambda_0)^\frac p2\ep^\frac p2>pv =pv-L(x,v)
	\end{align*}
	for small $\ep$, meaning that $v$ cannot be optimal in \eqref{eqo}.
	Hence, $v\in \op{int}(\C_x)$.
	
	Since $L$ is smooth on $\op{int}(\C)$, we can differentiate $w\mapsto pw-L(x,w)$ at its maximum point $w=v$ yielding
	\begin{align*}
	p=\frac{\partial L}{\partial v}(x,v).
	\end{align*}
	Therefore, using \eqref{eqk}, we get
	\begin{align*}
	H(x,p)=\frac{\partial L}{\partial v}(x,v)(v)-L(x,v)
	=
	p|v|_g^{p-2} (-|v|_g^2)+|v|_g^p
	=
	(1-p)|v|_g^p.
	\end{align*}
\end{proof}

\section{Semiconvexity}

\begin{lemma}[Criterion for uniform semiconvexity]\label{lema}
    Suppose that $(f_i)_i:U\to \R$ is a family of $C^1$-functions defined on the open subset $U\subseteq M$. Assume that there exists a compact set $K\subseteq T^*M$ such that 
    \[
        \{(x,d_xf_i)\mid x\in U\}\subseteq K.
    \]
    Moreover, let $h$ and $h_*$ be Riemannian metrics on $M$ and $T^*M$, respectively, and suppose that the family of functions
    \[
        df_i:U\to T^*M,\ x\mapsto (x,d_xf_i),
    \]
    is equi-Lipschitz. Then the family $(f_i)$ is both uniformly locally semiconcave and semiconvex.
\end{lemma}
\begin{proof}
    Fix $x_0\in U$ and let $\varphi:V\subseteq U\to \varphi(V)\subseteq \R^n$ be a smooth chart around $x_0$. For any $W\Subset V$, the mappings
    \[
        \varphi_{|W}^{-1}:\varphi(W)\to W,\quad \text{ and } \quad T^*\varphi:T^*W\cap K\to \R^n\times \R^n,
    \]
    are Lipschitz continuous (here, $T^*\phi$ denotes the cotangent bundle chart). It follows that the functions
    \[  
        \varphi(W)\to \R^n\times \R^n,\ x\mapsto (x,D(f_i\circ \varphi^{-1})(x)) = T^*\varphi \circ df_i\circ \varphi_{|W}^{-1}(x)
    \]
    are equi-Lipschitz. Using Taylor's theorem, this can easily be used to prove that the functions $(f_i\circ \varphi^{-1})$ are uniformly locally semiconcave and semiconvex.
\end{proof}

\section*{Acknowledgements}
I am grateful to Stefan Suhr for suggesting the study of a Lorentzian version of the Lasry–Lions regularization theorem.

\bibliography{OTaR_090625}
\end{document}